\documentclass[11pt,reqno]{amsart}
\usepackage{}
\usepackage{a4wide}
\numberwithin{equation}{section}
\usepackage{mathrsfs}
\usepackage{times}
\usepackage{amsfonts}
\usepackage{amsmath,extpfeil}
\usepackage{stmaryrd}
\usepackage{amssymb}
\usepackage{amsthm}
\usepackage{mathrsfs}
\usepackage{url}
\usepackage{amsfonts}
\usepackage{amscd}
\usepackage{indentfirst}
\usepackage{enumerate}
\usepackage{amsmath,amsfonts,amssymb,amsthm}
\usepackage{amsmath,amssymb,amsthm,amscd}
\usepackage{graphicx,mathrsfs}
\usepackage{appendix}
\usepackage[numbers,sort&compress]{natbib}
\usepackage{color}
\usepackage[colorlinks, linkcolor=blue, citecolor=blue]{hyperref}

\newtheorem{Thm}{Theorem}[section]
\newtheorem{Lem}{Lemma}[section]

\newtheorem{Prop}{Proposition}[section]

\newtheorem{Def}{Definition}[section]

\newtheorem*{remark*}{Remark}

\newtheorem*{proA}{Proposition A}
\newtheorem*{proB}{Proposition B}
\newtheorem*{proC}{Proposition C}
\newtheorem*{proD}{Proposition D}
\newtheorem*{proE}{Proposition E}

\allowdisplaybreaks[4]
\newcommand{\ep}{\varepsilon}
\newcommand{\vs}{\vspace}

\newcommand{\ra}{{\rightarrow}}

\begin{document}
\pagestyle{plain}
\title{Dirichlet problem for a class of nonlinear degenerate elliptic operators with critical growth and logarithmic perturbation}

\author[H. Chen, X. Liao and M. Zhang]{Hua Chen, Xin Liao and Ming Zhang}

\address[Hua Chen]{School of Mathematics and Statistics, Wuhan University, Wuhan 430070, China}
\email{chenhua@whu.edu.cn}

\address[Xin Liao]{School of Mathematics and Statistics, Wuhan University, Wuhan 430070, China}
\email{xin\_liao@whu.edu.cn}

\address[Ming Zhang]{School of Mathematics and Statistics, Wuhan University, Wuhan 430070, China}
\email{mingzhang\_math@whu.edu.cn}

\date{\today}

\begin{abstract}
In this paper, we  investigate  the existence of weak solutions for a class of degenerate elliptic Dirichlet problems with critical nonlinearity and a logarithmic perturbation, i.e.
\begin{equation}\label{eq0.1}
\Big\{\begin{array}{l}
-(\Delta_{x} u+(\alpha+1)^2|x|^{2 \alpha} \Delta_{y} u)=u^{\frac{Q+2}{Q-2}} + \lambda u\log u^2,\\
u=0~~ \text{ on } \partial \Omega,
\end{array}
\end{equation}
where $(x,y)\in\Omega \subset \mathbb{R}^N = \mathbb{R}^m \times \mathbb{R}^n$ with $m \geq 1$, $n\geq 0$, $\Omega \cap \{x=0\}\neq\emptyset$ is a bounded domain, the parameter  $\alpha \geq 0$ and $ Q=m+ n(\alpha +1)$ denotes the ‘‘homogeneous
dimension" of $\mathbb{R}^N$. When $\lambda=0$, we know that from \cite{Kogoj2012} the problem $(\ref{eq0.1})$ has a Poho\v{z}aev-type non-existence result. Then for $\lambda \in \mathbb{R}\backslash \{0\}$, we establish the existences of non-negative ground state weak solutions and non-trivial weak solutions subject to certain conditions.
\end{abstract}
\maketitle
\small
\keywords {\noindent {\bf Keywords:} {Degenerate elliptic operators, Dirichlet problem, critical nonlinearity, logarithmic perturbation, homogeneous
dimension.}
\smallskip
\vs{2mm}
\section{Introduction}\label{S1}
\vs{2mm}
\setcounter{equation}{0}

In this paper, we investigate the weak solutions $u=u(z)=u(x,y)$ of the following degenerate elliptic critical problem with logarithmic term
\begin{equation}\label{eq1.1}
\Big\{\begin{array}{l}
-(\Delta_{x} u+(\alpha+1)^2|x|^{2 \alpha} \Delta_{y} u)=u^{\frac{Q+2}{Q-2}} + \lambda u\log u^2,\\
u=0~~ \text{ on } \partial \Omega,
\end{array}
\end{equation}
where $z=(x,y)\in\Omega \subset \mathbb{R}^N = \mathbb{R}^m \times \mathbb{R}^n$ with $m \geq 1$, $n\geq 0$, $\Omega$ is a bounded domain with measure $|\Omega|$ that satisfies $\Omega \cap \{x=0\}\neq\emptyset$, the parameter  $\alpha \geq 0$, $\lambda \in \mathbb{R}\backslash \{0\}$ and $ Q=m+ n(\alpha +1)$ denotes the ‘‘homogeneous
dimension" of $\mathbb{R}^N$.
In addition, the partial differential operator
$\Delta_{\alpha}:=\Delta_{x} +(\alpha +1)^2|x|^{2 \alpha} \Delta_{y}$ is often called Baouendi-Grushin operator (cf. \cite{B1967}, \cite{G1970} and \cite{G1971}).

We denote by $\mathcal{D}_0^{1}(\Omega)$ the Sobolev space obtained as the completion of $C_0^{\infty}(\Omega)$ with respect to the norm
\begin{equation}\label{norm}
\|u\|_{\mathcal{D}_0^{1}(\Omega)}:=\Big(\int_{\Omega}(|\nabla_{x} u|^{2}+(\alpha+1)^2|x|^{2\alpha }|\nabla_{y} u|^{2}) \mathrm{d} x \mathrm{d} y\Big)^{\frac{1}{2} }=\Big(\int_{\Omega}|\nabla_{\alpha}u|^2 \mathrm{d} x \mathrm{d} y\Big)^{\frac{1}{2} }.
\end{equation}
Here, we use the notation $\nabla_{\alpha} := (\nabla_{x} , (\alpha+1)|x|^{\alpha}\nabla_{y})$ to denote the Grushin gradient. Thus, $\mathcal{D}_0^{1}(\Omega)$ is a Hilbert space endowed with the inner product
$\langle u,v\rangle=\int_{\Omega}\nabla_{\alpha}u\nabla_{\alpha}v \mathrm{d} x \mathrm{d} y$.

We emphasize that it is not necessary to assume $\Omega$ is bounded in the definition of $\mathcal{D}_0^{1}(\Omega)$. Furthermore, $u \in \mathcal{D}_0^{1}(\Omega)$ is called a weak solution to problem $\eqref{eq1.1}$ if for every $\varphi \in C_0^{\infty}(\Omega)$, we always have
\begin{equation}\label{weak-solution}
\int_{\Omega} \nabla_{\alpha}u \nabla_{\alpha}\varphi \mathrm{d} x \mathrm{d} y=\int_{\Omega}\varphi u^{\frac{Q+2}{Q-2}} + \lambda \varphi u\log u^2\mathrm{d} x \mathrm{d} y.
\end{equation}
In this case, we denote $2_{\alpha}^*:=\frac{2Q}{Q-2}$ as critical exponent, we call the problem \eqref{eq1.1} is critical because the embedding from $\mathcal{D}_0^{1}(\Omega)$ to $L^p(\Omega)$ is continuous and compact for $1<p<2_{\alpha}^*$ and is merely continuous when $p=2_{\alpha}^*$ (cf. \cite{Kogoj2012}). In the absence of compactness, we need to consider the extremal function for the following sharp Sobolev-type inequality (cf. \cite{M2006}):
\begin{equation}\label{sobolev}
    \begin{aligned}
        c_{m,n,\alpha}\|u\|_{L^{2_\alpha^*}(\mathbb{R}^N )} \leq \|\nabla_{\alpha} u\|_{L^{2}(\mathbb{R}^N )}, \quad \forall u \in C_{0}^{\infty}(\mathbb{R}^N ),
    \end{aligned}
\end{equation}
where $c_{m,n,\alpha}>0$ is the largest constant such that \eqref{sobolev} holds, i.e,
$$c_{m,n,\alpha}=\inf_{u\in    C_{0}^{\infty}(\mathbb{R}^N )}\{\|\nabla_{\alpha} u\|_{L^{2}(\mathbb{R}^N )}:\|u\|_{L^{2_\alpha^*}(\mathbb{R}^N )} =1 \}.$$
The extremal functions of \eqref{sobolev} are related to the following critical semilinear degenerate equation:
\begin{equation}\label{critical}
-(\Delta_{x} u+(\alpha+1)^2|x|^{2 \alpha} \Delta_{y} u)=u^{\frac{Q+2}{Q-2}},~~u\geq 0 ~\text{in } \mathbb{R}^N.
\end{equation}

When $\alpha=0,$ or $n=0$, the equation \eqref{critical} reduces to the constant scalar curvature equation on $\mathbb{R}^N,$ known as the Yamabe problem. In this case, the non-trivial solution or extremal function  is unique up to translation, scaling, and multiplication by a constant, and is given by the Talenti bubble (cf. \cite{Caffarelli1989}, \cite{GNN1981} and \cite{Talenti1976} )
\[U(z)=\frac{1}{(1+|z|^2)^{\frac{N-2}{2}}}.\]
For $n\geq1$ and $\alpha>0,$ the Baouendi-Grushin operator is known to be degenerate. In the special case where $n=1, m=2k ~(k\in \mathbb{N}^*)$ and $\alpha=1,$ the equation \eqref{critical} reduces to the constant Webster curvature equation on Heisenberg group for a solution $u(x,y)$ that is radially symmetric in the variable $x$.
When $\alpha=\frac{1}{2},$ the equation \eqref{critical}
represents the transonic flow problem, see, for example, \cite{Wang2003}.
For general $m, n,$ and $\alpha,$ Monti \cite{Monti2006} and \cite{M2006} proved that the extremal functions of \eqref{sobolev} possess some degree of ‘‘spherical symmetry" via the
Kelvin transforms methods.  Furthermore, Monti conjectured that  the
extremal functions of \eqref{sobolev} are
radial in the variable $x$.  In fact, Monti \cite{M2006} proved his conjecture partially when $m=1,$ his result is :
\begin{proA}\label{thmA}
  The infimum of the values
 \begin{equation*}
   \bar{c}_{m,n,\alpha}=\inf_{u\in    C_{0}^{\infty}(\mathbb{R}^N )}\{\|\nabla_{\alpha} u\|_{L^{2}(\mathbb{R}^N )}:\|u\|_{L^{2_\alpha^*}(\mathbb{R}^N )} =1 , u ~\text{is radial in the variable } x\}
 \end{equation*}
 are always achieved by  functions that are radial in the variable $x$. In the case where $m=1,$ we must have $c_{m,n,\alpha}=\bar{c}_{m,n,\alpha},$ thus implying $c_{m,n,\alpha}$ can also be achieved. Notably, when $m=n=\alpha=1,$ up to a translation in $y,$ scaling and multiplication by a nonzero constant, the extremal function for $(\ref{sobolev})$ takes the form of
 \[U(x,y)=\frac{1}{((1+x^2)^2+y^2)^{\frac{1}{4}}}.\]
\end{proA}
Another strong evidence in support of Monti's conjecture is presented in \cite{Be2001}, which gives
\begin{proB}\label{thmB}
  When $m=2,n=1, \alpha=1,$ up to a translation in $y,$ scaling and multiplication by a nonzero constant, the extremal function of $c_{2,1,1}$ is given by
 $$ U(x,y)=\frac{1}{\big((1+|x|^2)^2+y^2\big)^{\frac{1}{2}}}.$$
\end{proB}

 Additionally, Dou, Sun and Wang \cite{DSWZ2022} have  recently proved that
\begin{proC}\label{proA}
  If $\frac{2Q}{Q-2}$ is a positive integer, then
  $c_{m,n,\alpha}=\bar{c}_{m,n,\alpha},$ hence $c_{m,n,\alpha}$ can be achieved in this case.
\end{proC}
Furthermore, \cite{DSWZ2022} also provides an explicit formulation for the extremal functions of $\bar{c}_{m,n,\alpha}:$
\begin{proD}\label{thmC}
Assume that $m\neq 2$ or $m=2, n\neq 1.$
\begin{enumerate}
  \item For $\alpha=1,$ up to a translation in $y,$ scaling and multiplication by a nonzero constant, the extremal of $\bar{c}_{m,n,\alpha}$ is given by
      \[U(x,y)=\Big(\frac{1}{(1+|x|^2)^2+|y|^2}\Big)^{\frac{2n+m-2}{4}}.\]
  \item For $\alpha>0,$ up to a translation in $y,$ scaling and multiplication by a nonzero constant, the extremal of $\bar{c}_{m,n,\alpha}$ is given by
  \begin{equation}\label{Upsi}
  U(x,y)=\frac{1}{\big((|x|^{\alpha+1}+1)^2+|y|^2\big)^{\frac{Q-2}{2(\alpha+1)}}} \psi\Big(\big|\frac{(|x|^{\alpha+1}+1,y)}{(|x|^{\alpha+1}+1)^2+|y|^2}-
  (\frac{1}{2},0)\big|\Big),
\end{equation}
where $\psi(r)>0$ and $\psi\in C^2[0,\frac{1}{2})\cap C^0[0,\frac{1}{2}]$ is the unique solution to the ordinary differential equation
\begin{equation}\label{ode}
\Big\{\begin{array}{l}
\psi(r)''+(\frac{n}{r}-\frac{2\theta r}{\frac{1}{4}-r^2})\psi(r)' -\frac{\theta(n+\theta -1)}{\frac{1}{4}-r^2}\psi(r) =-C(\frac{1}{4}-r^2)^{\beta-\theta}\psi(r)^{\frac{n+2\beta-\theta+3}{n+\theta-1}}, ~0<r<\frac{1}{2}\\
\psi(\frac{1}{2})=1, \psi'(0)=0, \lim_{r\to (\frac{1}{2})^-}(\frac{1}{4}-r^2)^{\theta}\psi'(r)=0,
\end{array}
\end{equation}
for constant $C>0$ and $\theta=\frac{m+\alpha-1}{\alpha+1}, \beta =\frac{m}{\alpha+1}-1.$
\end{enumerate}
\end{proD}

Recently, nonlinear equations with logarithmic terms have attracted significant attention due to their broad applications in quantum mechanics, wave mechanics, nonlinear optics, nuclear physics, and other fields, see \cite{CG2018}, \cite{Z2010} and references therein. It is worth noting that Chen and Tian investigate in \cite{CT2015} the following semilinear pseudo-parabolic equations with logarithmic nonlinear terms
\begin{equation}\label{CT}
	\left\{	
	\begin{aligned}
		u_t -\Delta u_t-\Delta u&= u\log|u|,   \hspace{0.35cm} (x,t) \in \Omega \times (0,T), \\
		u(x,t)&=0,  \hspace{1.45cm} (x,t) \in \partial \Omega \times (0,T),\\
        u(x,0)&=u_0(x),  \hspace{0.8cm} x \in \Omega,
	\end{aligned}	
	\right.
\end{equation}
where $u_0 \in H_0^1(\Omega)$, $T \in (0,+\infty]$, $\Omega \subset \mathbb{R}^{N}$ is a bounded domain with smooth boundary $\partial\Omega$ and $N \geq 1$. Combined with the log-Sobolev inequality (see Section \ref{subsection5.2} below) and a family of potential wells relative to equation $(\ref{CT})$, the author finds the existence of global solution and vacuum isolating behavior of solutions. Furthermore,
we also recommend consulting the following references for more comprehensive information on sign-changing solution, positive solution and radial solution to logarithmic equations:
\cite{DM2014}, \cite{DPS2021}, \cite{S2023}, \cite{TZ2017}, \cite{T2016} and \cite{WZ2019}. On the other hand, in case of nonlinear degenerate elliptic equations with subcritical growth and general perturbation terms, the existence
and multiplicity of weak solutions to Dirichlet problems have been studied recently by \cite{Chen-Chen-Yuan2020,Chen-Chen-Yuan2022,Chen-Chen-Li-Liao2022, Chen-Chen-Li-Liao2022-1,Luyen2019}.

In case of $\alpha=0$, $\Delta_{0}=\Delta$ as the standard Laplacian and $Q=n+m=N$, a recent result for the problem \eqref{eq1.1} by Deng, He, Pan, et al. \cite{Deng2023} considered the elliptic critical problem with logarithmic perturbation, which used the property of Talenti bubbles and the methods of Br\'{e}zis-Nirenberg \cite{Brezis1983} to give
\begin{proE}\label{thmD}
In case of $\alpha=0$, $\lambda>0$ and $N\geq 4,$ then the problem \eqref{eq1.1} admits a positive ground state solution. When $\alpha=0$, $N=3$ and $-\frac{2c_{m,n,0}^3}{Q|\Omega|}<\lambda<0$, then the problem \eqref{eq1.1} admits a positive solution.
\end{proE}
Inspired by these articles, we establish the
existence of non-negative ground state solutions and non-trivial solutions to problem $(\ref{eq1.1})$ under  certain conditions. Actually, for general $\alpha>0,$  building on the conjecture of Monti and the results above for extremal functions, it is reasonable to assume that:
\begin{enumerate}[(A)]
  \item The value of $c_{m,n,\alpha}$ can be attained by a positive extremal function $U$.
\end{enumerate}

At least according to Propositions A-D, assumption (A) holds in the following three cases:
\begin{enumerate}
  \item When $m = 1$, $n$ is a natural number, and $\alpha \geq 0$.
  \item  When $m = 2, n = 1$, and $\alpha=1$.
  \item  When $(m, n)\neq(2,1)$, $\alpha \geq 0$, and the value of $\frac{2Q}{Q-2}$ is an integer.
\end{enumerate}

  We present our main results as follows:
\begin{Thm}\label{thm1.1}
  Assuming the condition (A), if $Q \geq 4$ and $\lambda >0$, then the problem $(\ref{eq1.1})$
   has a non-trivial non-negative ground state solution.
\end{Thm}
It is worth noting that the case for $Q = 4$ is particularly intricate.  In this scenario,
 as $\frac{2Q}{Q-2}=4$ is a positive integer, $c_{m,n,\alpha}$ is achieved by Proposition C, and the extremal functions can be explicitly formulated but rather complicated. Specifically, for $Q=m+ n(\alpha +1)=4$ and $\alpha>0$, we have three cases:
 $m=2,n=1, \alpha=1$;
 $m=1,n=2, \alpha=\frac{1}{2}$; and
 $m=1,n=1,\alpha=2$. In section \ref{s3.2} below, we will provide a comprehensive discussion of these cases in details.

Regarding the case for $\lambda<0,$ we will prove:
\begin{Thm}\label{thm1.2}
  Assuming the condition (A), if $2<Q<4$, $-\frac{2c_{m,n,\alpha}^{Q}}{Q|\Omega|}<\lambda <0$, then the problem $(\ref{eq1.1})$ has a non-trivial non-negative solution.
\end{Thm}
Furthermore, we apply the Dual Fountain Theorem (cf. \cite{BW1995}) to obtain the following result:
\begin{Thm}\label{thm1.3}
 Let $Q>2$, if $\lambda <0$, then the problem $(\ref{eq1.1})$ has a non-trivial weak solution. Moreover, if $0>\lambda \geq -\frac{2^{\frac{Q-2}{Q}} c_{m,n,\alpha}^2}{|\Omega|^{\frac{2}{Q}} (Q-2)^{\frac{Q-2}{Q}}},$ then the problem $(\ref{eq1.1})$ has infinitely many weak solutions with negative energy.
\end{Thm}

\section{Preliminaries}\label{S2}
\vs{2mm}
$\mathbf{Notations}$: For simplicity, different positive constants are commonly denoted by $C$, $C_i$, $i \in \mathbb{N}^+$, sometimes no designation is required. The notation $O(t)$ represents $|O(t)| \leq Ct$, and $o_n(1)$ denotes $o_n(1) \rightarrow 0$ as $n \rightarrow \infty$.

In addition, the integral symbol $\int_{\Omega} |u| \mathrm{d} x \mathrm{d} y$ and the $L^p(\Omega)$ norm $\|\cdot\|_{L^p(\Omega)}$ are often denoted by $\int_{\Omega} |u|\mathrm{d} z$ and $\|\cdot\|{p}$ respectively. The notation $u_+ = \max\{u,0\}$ and $u_- = \min\{u,0\}$ is also commonly used.

\subsection{Energy functional and mountain pass geometry structure}
\
\newline
 The energy functional $I:\mathcal{D}_0^{1}(\Omega) \rightarrow \mathbb{R}$ associated with equation $(\ref{eq1.1})$ is defined by
\begin{equation}\label{EF}
\begin{aligned}
I(u)=\frac{1}{2} \int_{\Omega}|\nabla_{\alpha} u|^{2}\mathrm{d} z-\frac{1}{2_\alpha^*} \int_{\Omega} u^{2_\alpha^*}\mathrm{d} z-\frac{\lambda}{2} \int_{\Omega}(u^{2} \log u^{2}-u^{2})\mathrm{d} z.
\end{aligned}
\end{equation}
Since $\Omega\subset \mathbb{R}^{N}$ is bounded and the following inequality holds, for $\delta>0$
\begin{equation}\label{b1.2}
	|u^{2}\log u^{2}|\leq C_{\delta} \big(|u|^{2-\delta}+|u|^{2+\delta}\big).
\end{equation}
Then from the Sobolev embedding in \cite{Kogoj2012}, one can easily verify that the functional $I(u)$ is well-defined in $\mathcal{D}_0^{1}(\Omega)$ and belongs to $\mathcal{C}^{1} \big(\mathcal{D}_0^{1}(\Omega), \mathbb{R}\big)$.
As a consequence, we get
\begin{equation}\label{a1.4}
	\langle I'(u), \varphi \rangle=\int_{\Omega}\nabla_{\alpha} u \nabla_{\alpha} \varphi \mathrm{d}z -\int_{\Omega}|u|^{2_\alpha^*-2}u\varphi \mathrm{d}z- \lambda \int_{\Omega}u\varphi\log u^{2}\mathrm{d}z,
\end{equation}
for any $u, \varphi \in \mathcal{D}_0^{1}(\Omega)$. Clearly, the critical points of $I$ are weak solutions to the problem $(\ref{eq1.1})$.
Furthermore, if we consider the energy functional
\begin{equation}\label{EF2}
\begin{aligned}
I_+(u)=\frac{1}{2} \int_{\Omega}|\nabla_{\alpha} u|^{2}\mathrm{d} z-\frac{1}{2_\alpha^*} \int_{\Omega} u_+^{2_\alpha^*}\mathrm{d} z-\frac{\lambda}{2} \int_{\Omega}(u_+^{2} \log u_+^{2}-u_+^{2})\mathrm{d} z.
\end{aligned}
\end{equation}
In turn, if $u$ is a critical point  of $I_+$,
we have
\begin{equation}\label{2.5}
	\langle I_+'(u), \varphi \rangle=\int_{\Omega}\nabla_{\alpha} u \nabla_{\alpha} \varphi \mathrm{d}z -\int_{\Omega}u_+^{2_\alpha^*-1}\varphi \mathrm{d}z- \lambda \int_{\Omega}u_+\varphi\log u_+^{2}\mathrm{d}z=0,
\end{equation}
for any $ \varphi \in \mathcal{D}_0^{1}(\Omega)$.
Since $u_+u_-=0$,
\[\int_{\Omega}u_+^{2_\alpha^*-1}u_- \mathrm{d}z= \int_{\Omega}u_+u_-\log u_+^{2}\mathrm{d}z=0.\]
Choose $\varphi=u_-$ in \eqref{2.5},
note that by  Lemma 3.5 in \cite{GN1996},
\begin{equation*}
\nabla_{\alpha} u_- =
  \begin{cases}
    \nabla_{\alpha} u, & \mbox{if } u<0, \\
    0, & \mbox{otherwise}.
  \end{cases}
\end{equation*}
One has
$$\int_{\Omega}\nabla_{\alpha} u \nabla_{\alpha} u_- \mathrm{d}z =\int_{\Omega}|\nabla_{\alpha} u_- |^2 \mathrm{d}z =0,$$
thus
$u_-=0$. Hence all the critical points of $I_+$ are non-negative solutions to the problem $(\ref{eq1.1}).$
\begin{Lem}\label{mpgs}
  If $Q>2, \lambda> -\frac{2c_{m,n,\alpha}^{Q}}{Q|\Omega|}$, then the functional $I_+$ satisfies the mountain pass geometry structure. That is,\\
\indent $(a)$~ $I_+(0)=0$,\\
\indent $(b)$~ there exist two constants $\sigma$, $\rho >0$ such that $I_+(u)>\sigma >0$ if $\|u\|_{\mathcal{D}_0^{1}(\Omega)}=\rho$,\\
\indent $(c)$~ there exists a function $v\in \mathcal{D}_0^1(\Omega)$ so that $I_+(v)\leq 0$ if $\|v\|_{\mathcal{D}_0^{1}(\Omega)}>\rho$,\\
where $\|u\|_{\mathcal{D}_0^{1}(\Omega)}$ is defined by $(\ref{norm})$.
\end{Lem}
\begin{proof}
Clearly, $I_+(0)=0$. To verify conditions $(b)$ and $(c)$, we divide the proof into two cases.

{\bf{Case 1: $\lambda \geq 0$.}}

Using the sharp Sobolev-type inequality given in $(\ref{sobolev})$, for $\lambda \geq 0$ and $0<\delta <\min\{2, 2_\alpha^{*}-2\}$, there exists $\sigma >0$ such that
\begin{equation*}
\begin{aligned}
  I_+(u)
  &=\frac{1}{2}\|u\|_{\mathcal{D}_0^{1}(\Omega)}^{2}-\frac{1}{2_\alpha^{*}} \|u_+\|_{2_\alpha^{*}}^{2_\alpha^{*}}+\frac{\lambda}{2}\|u_+\|_{2}^{2} -\frac{\lambda}{2}\int_{\Omega} u_+^{2}\log u_+^{2} \mathrm{d}z \\
  &\geq \frac{1}{2}\|u\|_{\mathcal{D}_0^{1}(\Omega)}^{2}-\frac{1}{2_\alpha^{*}} \|u_+\|_{2_\alpha^{*}}^{2_\alpha^{*}}+\frac{\lambda}{2} \|u_+\|_{2}^{2}-\frac{\lambda}{2}\int_{\{|u_+|\geq 1\}}u_+^{2} \log u_+^{2} \mathrm{d}z \\
  &\geq \frac{1}{2}\|u\|_{\mathcal{D}_0^{1}(\Omega)}^{2}-C\|u_+\|_{\mathcal{D}_0^{1}(\Omega)}^{2_\alpha^{*}}-C\|u_+\|_{\mathcal{D}_0^{1}(\Omega)}^{2+\delta}\geq \sigma>0
\end{aligned}
\end{equation*}
if $\| u\|=\rho>0$ is sufficiently small.

For a non-zero function $u\in \mathcal{D}_0^1(\Omega)$ and $u \geq 0$ a.e., set $v:=tu$ with $t>0$. Then we get,
\begin{equation*}
I_+(v)=I_+(tu)= \frac{t^{2}}{2}\|u\|_{\mathcal{D}_0^{1}(\Omega)}^{2}-\frac{t^{2_\alpha^{*}}}{2_\alpha^{*}} \|u_+\|_{2_\alpha^{*}}^{2_\alpha^{*}}+\frac{\lambda t^{2}}{2}\|u_+\|_{2}^{2} -\frac{t^{2}}{2} \int_{\Omega} \lambda u_+^{2} \log (t^{2} u_+^{2}) \mathrm{d} z \leq 0
\end{equation*}
if $t>0$ is large enough.

{\bf{Case 2: $-\frac{2c_{m,n,\alpha}^{Q}}{Q|\Omega|} < \lambda < 0$.}}

First of all, for $\lambda < 0$, we have
\begin{equation*}
\begin{aligned}
I_+(u)
  &=\frac{1}{2}\|u\|_{\mathcal{D}_0^{1}(\Omega)}^{2}-\frac{1}{2_\alpha^{*}} \|u_+\|_{2_\alpha^{*}}^{2_\alpha^{*}}-\frac{\lambda}{2} \int u_{+}^{2}(\log u_{+}^{2}-1) \mathrm{d} z\\
  &= \frac{1}{2}\|u\|_{\mathcal{D}_0^{1}(\Omega)}^{2}-\frac{1}{2_\alpha^{*}} \|u_+\|_{2_\alpha^{*}}^{2_\alpha^{*}}-\frac{\lambda}{2} \int u_{+}^{2} \log (e^{-1} u_{+}^{2}) \mathrm{d} z\\
  &\geq \frac{1}{2}\|u\|_{\mathcal{D}_0^{1}(\Omega)}^{2}-\frac{1}{2_\alpha^{*}} \|u_+\|_{2_\alpha^{*}}^{2_\alpha^{*}}-\frac{\lambda}{2} \int_{\{e^{-1} u_{+}^{2} \leq 1\}} u_{+}^{2} \log (e^{-1} u_{+}^{2}) \mathrm{d} z\\
  &\geq \frac{1}{2}\|u\|_{\mathcal{D}_0^{1}(\Omega)}^{2}-\frac{1}{2_\alpha^{*}} \|u_+\|_{2_\alpha^{*}}^{2_\alpha^{*}}-\frac{\lambda}{2}\int_{\{e^{-1} u_{+}^{2} \leq 1\}}(-1) \mathrm{d} z\\
  &\geq \frac{1}{2}\|u\|_{\mathcal{D}_0^{1}(\Omega)}^{2}-\frac{1}{2_\alpha^{*}}c_{m,n,\alpha}^{-2_\alpha^{*}} \|u\|_{\mathcal{D}_0^{1}(\Omega)}^{2_\alpha^{*}}+ \frac{\lambda}{2}|\Omega|.
\end{aligned}
\end{equation*}
Condition $(b)$ is satisfied if
$$\frac{1}{2}\|u\|_{\mathcal{D}_0^{1}(\Omega)}^{2}-\frac{1}{2_\alpha^*}c_{m,n,\alpha}^{-2_\alpha^{*}} \|u\|_{\mathcal{D}_0^{1}(\Omega)}^{2_\alpha^{*}}+ \frac{\lambda}{2}|\Omega| >0 \ \ \text{when} \ \ \|u\|_{\mathcal{D}_0^{1}(\Omega)} = c_{m,n,\alpha}^{\frac{Q}{2}},$$
which means that $-\frac{2c_{m,n,\alpha}^{Q}}{Q|\Omega|} < \lambda < 0$ and is consistent with our hypothesis. In addition, applying a similar argument as Case 1 above, we can find a function $v\in \mathcal{D}_0^1(\Omega)$ so that $I_+(v)\leq 0$ if $\|v\|_{\mathcal{D}_0^{1}(\Omega)}>\rho$. Then we complete the proof.
\end{proof}

\subsection{The Palais-Smale condition}
\begin{Def}On a Hilbert space $H$,
  We call $\{u_n\}_{n=1}^{\infty}$ is a $(PS)_c$ sequence of a functional
  $E \in \mathcal{C}^1(H,\mathbb{R})$,
  if $ E(u_n) \rightarrow c, E'(u_n) \rightarrow 0 ~ \text{as}~ n \rightarrow \infty.$ We call $\{u_n\}_{n=1}^{\infty}$ is a $(PS)_c^*$ sequence of a functional
  $E \in \mathcal{C}^1(H,\mathbb{R})$, if there exist a sequence of finite dimensional space $H_n$ with $H_n\subset H_{n+1} \to H$ such that
  $u_n\in H_{n}, E(u_n)\to c, (E|_{H_n})'(u_n) \rightarrow 0 ~ \text{as}~ n \rightarrow \infty . $
  If every $(PS)_c$ or $(PS)_c^*$ sequence is bounded and admits a convergent subsequence, then we say the functional $E$ satisfies the $(PS)_c$ or $(PS)_c^*$ condition. Besides, it is worth noting that $(PS)_c^*$ condition implies $(PS)$ condition (cf. Remarks 3.19 in \cite{W1997}).
 \end{Def}
\begin{Lem}\label{log.sl}
	Assume that there exists a sequence $\{u_{n}\}\subset \mathcal{D}_0^1(\Omega)$ satisfying $u_{n}\rightharpoonup u$ weakly in $\mathcal{D}_0^1(\Omega)$ and $u_n \to u$ a.e. in $\Omega$. Then we obtain
	\begin{equation}\label{e4}
		{\lim_{n \to +\infty}}\int_{\Omega}{u_{n}^{2}}\log u_n^{2} \mathrm{d} z=\int_{\Omega}{u^{2}}\log u^{2} \mathrm{d} z.
	\end{equation}
\end{Lem}

\begin{proof}
	 For $0<\delta  <\min\{2, 2_\alpha^{*}-2\}$, and any Lebesgue measurable set $A\subset \Omega$ we have from $(\ref{b1.2})$ that
	\begin{equation*}
		\int_{A}\big|{u_{n}^{2}}\log u_{n}^{2}\big|dx\leq C_{\delta } \int_{A}\big(|u_{n}|^{2-\delta}+|u_{n}|^{2+\delta}\big)dx,
	\end{equation*}
	where $ C_{\delta}>0$ is a constant dependent on $\delta$. From the Vitali convergence theorem, we can deduce  $(\ref{e4})$ from the fact $u_n \to u$ in $L^{2\pm \delta}(\Omega)$.
\end{proof}
The following proposition shows us the energy functional and the convergence of $(PS)_c^*$ sequences have the following properties.
\begin{Prop}\label{bounded}For $Q>2,$ we have:
\begin{enumerate}[$(a)$]
\item If $\lambda \geq 0$ and $c < \frac{1}{Q}c_{m,n,\alpha}^Q$, then $I(u)$ and $I_+(u)$ satisfy the $(PS)_c^*$ condition;
\item If $\lambda <0$ and $c < \frac{1}{Q}c_{m,n,\alpha}^Q -\frac{|\Omega|}{Q}\big(\frac{Q-2}{2}\big)^{\frac{Q-2}{2}}|\lambda|^{\frac{Q}{2}}  $, then $I(u)$ and $I_+(u)$ satisfy the $(PS)_c^*$  condition;
\item If $\lambda<0$ and $c\in (-\infty, 0) \cup (0,\frac{1}{Q}c_{m,n,\alpha}^Q),$ then for every $(PS)_c^*$ sequence $\{u_n\}$ of $I$ or $I_+$, there exists $0\neq u\in \mathcal{D}_0^1(\Omega)$ such that $u_n\rightharpoonup u$ weakly and $u$ is a critical point.
\end{enumerate}
\end{Prop}
\begin{proof}Without loss of generality, we only prove the case for $I$. The proof for $I_+$ will be similar.

{{Proof of $(a)$}:}
 Let $\{u_n\}\subset H_n$ be a $(PS)_c^*$ sequence of $I$, where $H_n$ is a finite dimensional space with $H_n\subset H_{n+1} \to \mathcal{D}_0^{1}(\Omega)$, which means that
\begin{equation}\label{e0}
  I(u_n) \rightarrow c,\ \ (I|_{H_n})'(u_n) \rightarrow 0 ~ \text{as}~ n \rightarrow \infty.
\end{equation}

We first claim that $\|u_n\|$ is bounded. In fact, from $(\ref{EF})$ and $(\ref{a1.4})$, we obtain
\begin{equation}\label{e1}
  \frac{1}{2} \int_{\Omega}|\nabla_{\alpha} u_n|^{2}\mathrm{d} z-\frac{1}{2_\alpha^*} \int_{\Omega} u_n^{2_\alpha^*}\mathrm{d} z-\frac{\lambda}{2} \int_{\Omega} (u_n^{2} \log u_n^{2}-u_n^{2})\mathrm{d} z=c + o_n(1),
\end{equation}
and
\begin{equation}\label{e2}
 \int_{\Omega}|\nabla_{\alpha} u_n|^2 \mathrm{d}z -\int_{\Omega}u_n^{2_\alpha^*} \mathrm{d}z- \lambda \int_{\Omega}u_n^{2}\log u_n^{2}\mathrm{d}z = o_n(1)\|u_n\|_{\mathcal{D}_0^{1}(\Omega)}
\end{equation}
as $n \rightarrow \infty$. Then it follows from $(\ref{e1})$ and $(\ref{e2})$ that
\begin{equation*}
\begin{aligned}
  c + o_n(1)+ o_n(1)\|u_n\|_{\mathcal{D}_0^{1}(\Omega)} &= I(u_n)-\frac{1}{2}\langle I'(u_n), u_n \rangle \\
  &= (\frac{1}{2}-\frac{1}{2_\alpha^*})\|u_n\|_{2_\alpha^*}^{2_\alpha^*} + \frac{\lambda}{2} \|u_n\|_{2}^{2} \geq C\|u_n\|_{2_\alpha^*}^{2_\alpha^*}-C \ \ \text{as} \ n \rightarrow \infty.
\end{aligned}
\end{equation*}
That means $\|u_n\|_{2_\alpha^*}^{2_\alpha^*} \leq C + o_n(1)\|u_n\|_{\mathcal{D}_0^{1}(\Omega)}$. Using $(\ref{e2})$ again, we get, for $n$ large enough,
\begin{equation}\label{e3}
\begin{aligned}
\|u_n\|_{\mathcal{D}_0^{1}(\Omega)}^2 &\leq C + o_n(1)\|u_n\|_{\mathcal{D}_0^{1}(\Omega)} + \lambda \int_{\Omega}u_n^{2}\log u_n^{2}\mathrm{d} z \\
&\leq C + o_n(1)\|u_n\|_{\mathcal{D}_0^{1}(\Omega)} +C \int_{\Omega}u_n^{2_\alpha^*} \mathrm{d}z + C \int_{\Omega}|u_n| \mathrm{d}z\\
&\leq C + o_n(1)\|u_n\|_{\mathcal{D}_0^{1}(\Omega)} +C \|u_n\|_{2_\alpha^*}^{2_\alpha^*} +C\|u_n\|_{2_\alpha^*}\\
&\leq C + o_n(1)\|u_n\|_{\mathcal{D}_0^{1}(\Omega)} + o_n(1)\|u_n\|_{\mathcal{D}_0^{1}(\Omega)}^{\frac{1}{2_\alpha^*}}.
\end{aligned}
\end{equation}
Then there exists $C>0$ such that $\|u_n\|_{\mathcal{D}_0^{1}(\Omega)} \leq C$, which means that $\{u_n\}$ is bounded in $H_n$. So up to a subsequence, there exists $u \in \mathcal{D}_0^1(\Omega)$ such that
\begin{equation*}
\begin{aligned}
&u_n \rightharpoonup u \quad \mbox{weakly in} ~ \mathcal{D}_0^1(\Omega),\\
&u_n \rightarrow u \quad \mbox{strongly in} ~ L^p (\Omega), \  1\leq p < 2^*_{\alpha},\\
&u_n \rightarrow u \quad \mbox{a.e.} \ \mbox{in} \ \Omega.
\end{aligned}
\end{equation*}
Since $(I|_{H_n})'(u_n) \rightarrow 0,$ then for every $\varphi_k\in H_k,$ we have
 \begin{equation*}
\lim_{n\to\infty} \langle I'(u_n), \varphi_k \rangle=\langle I'(u), \varphi_k \rangle=0.
 \end{equation*}
 For $\varphi \in \mathcal{D}_0^1(\Omega)$ and any $\varepsilon>0,$ and $k_0$ large enough, there exists  $\varphi_{k_0}\in H_{k_0}$ such that $\|\varphi-\varphi_{k_0}\|_{\mathcal{D}_0^1(\Omega)}\leq \varepsilon,$ thus
 \[\langle I'(u), \varphi \rangle =\langle I'(u), \varphi_{k_0} \rangle+ \langle I'(u), \varphi-\varphi_{k_0} \rangle\leq C\varepsilon.\]
 Let $\varepsilon\to0,$ we obtain $\langle I'(u), \varphi \rangle =0$.
That implies $u$ is a weak solution of the problem $(\ref{eq1.1})$. Then it follows that
 \begin{equation}\label{q1}
   I(u)=\big(\frac{1}{2}-\frac{1}{2_{\alpha}^{*}}\big)\int_{\Omega} u^{2_{\alpha}^{*}}\mathrm{d} z + \frac{\lambda}{2} \int_{\Omega}u^{2}\mathrm{d} z.
 \end{equation}
Since $\langle I'(u_n), u_n \rangle \rightarrow 0$, from Lemma $\ref{log.sl}$ and  $(\ref{e0})$--$(\ref{e2})$, we may set $v_n=u_n-u$, then we have
\begin{equation}\label{e7.2}
  I(u) + \frac{1}{2}\int_{\Omega}|\nabla_{\alpha} v_n|^{2}\mathrm{d} z - \frac{1}{2_\alpha^*} \int_{\Omega}v_n^{2_{\alpha}^{*}}\mathrm{d} z = c + o_n(1)
\end{equation}
and
\begin{equation*}
  \int_{\Omega}|\nabla_{\alpha} v_n|^{2}\mathrm{d} z - \int_{\Omega}v_n^{2_{\alpha}^{*}}\mathrm{d} z =o_n(1).
\end{equation*}
We may therefore assume that
$$\|\nabla_{\alpha} v_n\|_2^2 \to l, \ \ \|v_n\|_{2_{\alpha}^{*}}^{2_{\alpha}^{*}}\to l \ \ \text{as} \ n \rightarrow \infty.$$
Using Sobolev-type inequality given in $(\ref{sobolev})$, we deduce
\begin{equation}\label{e8}
\begin{aligned}
l \geq c_{m,n,\alpha}^2 l^{2/{2_{\alpha}^{*}}}.
\end{aligned}
\end{equation}
If $l=0$, then the proof is complete. If not, $l \geq c_{m,n,\alpha}^Q$. Combined it with $(\ref{q1})$ and $(\ref{e7.2})$, we get for $\lambda \geq 0$
\begin{equation}\label{f000}
\begin{aligned}
c &= \big(\frac{1}{2}-\frac{1}{2_{\alpha}^{*}}\big)(l + \|u\|_{2_{\alpha}^{*}}^{2_{\alpha}^{*}}) + \frac{\lambda}{2}\|u\|_2^2\\
&\geq \frac{1}{Q}(c_{m,n,\alpha}^Q + \|u\|_{2_{\alpha}^{*}}^{2_{\alpha}^{*}})
\geq \frac{1}{Q}c_{m,n,\alpha}^Q,
\end{aligned}
\end{equation}
which contradicts with $c < \frac{1}{Q}c_{m,n,\alpha}^Q$. Thus $l=0$, $u_n \rightarrow u$ strongly in $\mathcal{D}_0^1(\Omega)$ and $I(u)$ satisfies the $(PS)_c^*$ condition.

{{Proof of $(b)$}:}
On the other hand, for $\lambda < 0$ and $l\not=0$, we deduce from $(\ref{f000})$ that
\begin{equation*}
  c \geq \frac{1}{Q}c_{m,n,\alpha}^Q + \frac{1}{Q} \|u\|_{2_{\alpha}^{*}}^{2_{\alpha}^{*}} - \frac{|\lambda|}{2}|\Omega|^{1-\frac{2}{2_{\alpha}^{*}}}\|u\|_{2_{\alpha}^{*}}^2.
\end{equation*}
For $g(t)=\big(\frac{1}{Q} t^{2_{\alpha}^{*}} - \frac{|\lambda|}{2}|\Omega|^{1-\frac{2}{2_{\alpha}^{*}}} t^2\big)$, then $g'(t)|_{t=t_0}=0$ implies that $t_0=\big(\frac{Q-2}{2}\big)^{\frac{Q-2}{4}}|\lambda|^{\frac{Q-2}{4}}|\Omega|^{\frac{1}
{2^*_\alpha}}$. From the direct calculation, we notice that
$$\min_{t>0} \big(\frac{1}{Q} t^{2_{\alpha}^{*}} - \frac{|\lambda|}{2}|\Omega|^{1-\frac{2}{2_{\alpha}^{*}}} t^2\big)= g(t_0)=- \frac{|\Omega|}{Q}\big(\frac{Q-2}{2}\big)^{\frac{Q-2}{2}}|\lambda|^{\frac{Q}{2}}.$$
Then we deduce that
\begin{equation*}
  c \geq \frac{1}{Q}c_{m,n,\alpha}^Q - \frac{|\Omega|}{Q}\big(\frac{Q-2}{2}\big)^{\frac{Q-2}{2}}|\lambda|^{\frac{Q}{2}},
\end{equation*}
which contradicts to the condition in $(b)$: $c< \frac{1}{Q}c_{m,n,\alpha}^Q - \frac{|\Omega|}{Q}\big(\frac{Q-2}{2}\big)^{\frac{Q-2}{2}}|\lambda|^{\frac{Q}{2}}$. Thus we have $l=0$ and the conclusion of $(b)$ is proved.

{{Proof of $(c)$}:}
 Suppose $\{u_n\}$ is a  $(PS)_c^*$ sequence.
 From the proof of $(a)$, we know $\{u_n\}$ is bounded in $\mathcal{D}_0^1(\Omega).$ Thus, up to a subsequence, there exists $u\in \mathcal{D}_0^1(\Omega)$ such that
\begin{equation*}
\begin{aligned}
&u_n \rightharpoonup u \quad \mbox{weakly in} ~ \mathcal{D}_0^1(\Omega),\\
&u_n \rightarrow u \quad \mbox{strongly in} ~ L^p (\Omega), \  1\leq p < 2^*_{\alpha},\\
&u_n \rightarrow u \quad \mbox{a.e.} \ \mbox{in} \ \Omega.
\end{aligned}
\end{equation*}
Similar to the proof of $(a)$, we can also deduce $u$ is a weak solution of the problem $(\ref{eq1.1})$. Assume that $u = 0$. Then from the condition in $(c)$, we can get from $(\ref{e8})$ and $(\ref{f000})$ that $l \geq c_{m,n,\alpha}^2 l^{2/{2_{\alpha}^{*}}}$ and $c=\big(\frac{1}{2}-\frac{1}{2_{\alpha}^{*}}\big)l\neq 0$. That means $c \geq \frac{1}{Q}c_{m,n,\alpha}^Q$, which is a contradiction with the condition in $(c)$.  Therefore, we have $u \neq 0$.
\end{proof}

\subsection{Basic estimates for rescaled extremal function}
\
\newline
\indent For $z=(x,y) \in \Omega$, let $ d_\alpha(z)=(|x|^{2(\alpha+1)}+|y|^{2})^{\frac{1}{2(\alpha+1)}}$ denote the homogeneous metric in $\mathbb{R}^N$, and define
$B_r(z)=\{v \in \mathbb{R}^N \big{|} d_\alpha(z-v)<r\}.$
Hence, we shall employ the formula for radial functions in polar coordinates on homogeneous groups, which we recall here: For every $0\leq r_1<r_2$ and for every measurable function $f:[r_1,r_2]\rightarrow \mathbb{R}$, we have
\begin{equation}\label{polar}
\int_{B_{r_2}(0)\backslash B_{r_1}(0)} f(d_\alpha(z))\mathrm{d} z=Q|B_1(0)|\int_{r_1}^{r_2}f(\rho)\rho^{Q-1}\mathrm{d}\rho,
\end{equation}
provided at least one of the two integrals exists.

If $c_{m,n,\alpha}$ in $(\ref{sobolev})$ can be attained by a positive function $U$, then up to scaling, and multiplication by a positive constant, we can assume that
\begin{equation}\label{Scale}
\|\nabla_{\alpha} U\|_{2}^{2}=\|U\|_{2_\alpha^*}^{2_\alpha^*}=c_{m,n,\alpha}^{Q}.
\end{equation}
Correspondingly, $U$ satisfies $-\Delta_{\alpha} U= U^{2_{\alpha}^*-1}$ in $\mathbb{R}^N$. Furthermore,  we have:

\begin{Lem}\label{lem2.5}
For $U(z)$ mentioned in \eqref{Scale}, we obtain
\begin{equation}\label{a1.3}
    \begin{aligned}
        U(z) \simeq \frac{1}{d_\alpha(z)^{Q-2}}~\text { as } ~d_\alpha(z) \rightarrow +\infty.
    \end{aligned}
\end{equation}
That is, there exist constants $C_1, C_2$ such that
\[ \frac{C_1}{d_\alpha(z)^{Q-2}}\leq U(z) \leq \frac{C_2}{d_\alpha(z)^{Q-2}},~\text{ when }~ d_\alpha(z)>1.\]
\end{Lem}
\begin{proof}
  See  Theorem 1.1 in \cite{L2009}.
\end{proof}
We emphasize here that by employing localization techniques and utilizing the same regular estimation approach as outlined in the appendix (see section 5 below), one can demonstrate that $U$ is continuous.

For $\ep>0$, the rescaled functions is defined as
$$U_{\ep}(x,y) = \ep^{\frac{2-Q}{2}}U(\frac{1}{\ep} x, \frac{1}{\ep^{\alpha +1}}y),$$
which satisfies $-\Delta_{\alpha} U_{\ep}= U_{\ep}^{2_{\alpha}^*-1}$ in $\mathbb{R}^N$ and
$$\|\nabla_{\alpha} U_{\ep}\|_{2}^{2}=\|U_{\ep}\|_{2_\alpha^*}^{2_\alpha^*}=c_{m,n,\alpha}^{Q},~\text { for all } \ep>0.$$

Since $\Omega \cap \{x=0\} \neq \emptyset$, we may
assume $0\in \Omega$ without loss of generality.
Take suitable small $R>0$ satisfying
$$B_R(0) \subset B_{2R}(0) \subset \Omega,$$
and $\eta(x,y) \in C_{0}^{\infty}(\Omega)$ satisfying $\eta(x,y)=1$ in $B_R(0)$, $0 \leq \eta(x,y) \leq 1$ in $B_{2R}(0)$, $\text{supp} (\eta(x,y)) \subset B_{2R}(0)$, then we set
\begin{equation}\label{a1.6}
    \begin{aligned}
      u_{\ep}(x,y):=\eta(x,y)U_{\ep}(x,y).
    \end{aligned}
\end{equation}

We have the following estimates immediately
\begin{Lem}\label{gt}
  For $u_{\ep}$ defined by $(\ref{a1.6})$ and $Q>2$, when $\ep\to 0^+$, we have
\begin{align}
\int_{\Omega}|\nabla_{\alpha} u_\ep|^{2}\mathrm{d} z&=c_{m,n,\alpha}^{Q}+O(\varepsilon^{Q-2}), \label{a1.71}\\
\int_{\Omega}u_\ep^{2_\alpha^*}\mathrm{d} z&=c_{m,n,\alpha}^{Q}+O(\varepsilon^{Q}), \label{a1.72}\\
\int_{\Omega}u_\ep^{2}\mathrm{d} z &\leq\left\{\begin{array}{ll}
C \varepsilon^{Q-2}, & \text { if } ~2<Q<4, \\
C\ep^2|\log\ep|+O(\ep^2),& \text { if } ~Q=4,\\
C \varepsilon^{2}+O(\varepsilon^{Q-2}), & \text { if }~ Q>4.
\end{array}\label{a1.73}\right.
\end{align}
All the constants $C$ here are independent of $R$ and $\ep$.
\end{Lem}
\begin{proof}
  Let us begin to compute
\begin{equation*}
    \begin{aligned}
\int_{\Omega}|\nabla_{\alpha} u_\ep|^{2}\mathrm{d} z&=\int_{\Omega}|\nabla_{\alpha}(\eta U_{\varepsilon})|^{2}\mathrm{d} z=\int_{\Omega}|\eta \nabla_{\alpha} U_{\varepsilon}+U_{\varepsilon} \nabla_{\alpha} \eta|^{2}\mathrm{d} z \\
&=\int_{\Omega} \eta^{2}|\nabla_{\alpha} U_{\varepsilon}|^{2}\mathrm{d} z+2 \int_{\Omega} ( \eta U_{\varepsilon})\langle\nabla_{\alpha} \eta, \nabla_{\alpha} U_{\varepsilon}\rangle\mathrm{d} z+\int_{\Omega} U_{\varepsilon}^{2}|\nabla_{\alpha} \eta|^{2}\mathrm{d} z\\
&=\int_{\Omega}\langle\nabla_{\alpha} U_{\varepsilon} ,\nabla_{\alpha}(\eta^{2} U_{\varepsilon})\rangle\mathrm{d} z+\int_{\Omega} U_{\varepsilon}^{2}|\nabla_{\alpha} \eta|^{2}\mathrm{d} z =\int_{\Omega} \eta^{2} U_{\varepsilon}^{2_\alpha^*}\mathrm{d} z+\int_{\Omega} U_{\varepsilon}^{2}|\nabla_{\alpha} \eta|^{2}\mathrm{d} z \\
&=\int_{\mathbb{R}^{N}} U_{\varepsilon}^{2_\alpha^*}\mathrm{d} z+\int_{\Omega} U_{\varepsilon}^{2}|\nabla_{\alpha} \eta|^{2}\mathrm{d} z-\int_{\mathbb{R}^{N}\backslash\Omega}U_{\varepsilon}^{2_\alpha^*}\mathrm{d} z-\int_{\Omega}(1-\eta^{2}) U_{\varepsilon}^{2_\alpha^*}\mathrm{d} z\\
&:=\int_{\mathbb{R}^{N}} U_{\varepsilon}^{2_\alpha^*}\mathrm{d} z+A_{1}-A_{2},
\end{aligned}
\end{equation*}
where
$$A_{1}=\int_{\Omega} U_{\varepsilon}^{2}|\nabla_{\alpha} \eta|^{2}\mathrm{d} z,\ \ \ \ A_{2}=\int_{\mathbb{R}^{N}\backslash\Omega}U_{\varepsilon}^{2_\alpha^*} \mathrm{d} z+ \int_{\Omega}(1-\eta^{2}) U_{\varepsilon}^{2_\alpha^*} \mathrm{d} z.$$
It can be determined that $\nabla_{\alpha} \eta=0$ in $B_R(0)$ when $\eta=1$. Therefore, we have
\begin{equation*}
    \begin{aligned}
A_{1} & \leq C \int_{B_{2 R}(0) \backslash B_{R}(0)} U_{\varepsilon}^{2}\mathrm{d} z =C \int_{B_{2 R}(0) \backslash B_{R}(0)} \varepsilon^{2-Q} U^{2}(\frac{1}{\varepsilon} x, \frac{1}{\varepsilon^{\alpha+1}}y)\mathrm{d} z \\
&=C \int_{B_{\frac{2R}{\ep}}(0) \backslash B_{\frac{R}{\ep}}(0)} \varepsilon^{2-Q}\varepsilon^{Q} U^2\mathrm{d} z \leq C \int_{\frac{R}{\ep}<d_\alpha(z)<\frac{2R}{\ep}} \varepsilon^{2} \frac{1}{d_\alpha(z)^{2 Q-4}} \mathrm{d} z\\
&=C \int_{\frac{R}{\ep}}^{\frac{2R}{\ep}} \varepsilon^{2} r^{3-Q} \mathrm{d}r =O(\varepsilon^{Q-2}).
\end{aligned}
\end{equation*}
Moreover, it is easily seen that $ A_2=O(\ep^Q)$. Indeed,
\begin{equation*}
    \begin{aligned}
0 & \leq \int_{\Omega}(1-\eta^{2}) U_{\varepsilon}^{2_\alpha^*}\mathrm{d} z \leq \int_{\Omega \backslash B_{ R}(0)} U_{\varepsilon}^{2_\alpha^*}\mathrm{d} z =\int_{\Omega \backslash B_{ R}(0)} \varepsilon^{-Q} U^{2_\alpha^*}\big(\frac{1}{\varepsilon} x, \frac{1}{\varepsilon^{\alpha+1}}y\big)\mathrm{d} z\\
&\leq \int_{\mathbb{R}^N \backslash B_{\frac{ R}{\ep}}(0)} \ep^{-Q}\ep^{Q} U^{2_\alpha^*} \mathrm{d} z \leq \int_{d_{\alpha}(z)>\frac{ R}{\ep}} \frac{C}{d_\alpha(z)^{2 Q}}\mathrm{d} z\\
&=C \int_{\frac{ R}{\ep}}^{+\infty} r^{-Q-1} \mathrm{d} r=O(\varepsilon^{Q}),
\end{aligned}
\end{equation*}
and an equivalent estimate applies to the other term in $A_2$. So we obtain
\begin{equation}\label{a56}
  \int_{\Omega}|\nabla_{\alpha} u_\ep|^{2}\mathrm{d}z=c_{m,n,\alpha}^{Q}+O(\varepsilon^{Q-2}) -O(\varepsilon^{Q})=c_{m,n,\alpha}^{Q}+O(\varepsilon^{Q-2}).
\end{equation}
Concerning $\int_{\Omega}u_\ep^{2_\alpha^*}\mathrm{d} z$, we get
\begin{equation*}
\begin{aligned}
\int_{\Omega}|u_\ep|^{2_\alpha^*}\mathrm{d} z &=\int_{\Omega}(\eta U_{\varepsilon})^{2_\alpha^*}\mathrm{d}z =\int_{\mathbb{R}^{N}} U_{\varepsilon}^{2_\alpha^*}\mathrm{d} z-\int_{\mathbb{R}^{N} \backslash \Omega} U_{\varepsilon}^{2_\alpha^*}\mathrm{d} z+\int_{\Omega}(\eta^{2_\alpha^*}-1) U_{\varepsilon}^{2_\alpha^*}\mathrm{d} z=c_{m,n,\alpha}^{Q}+O(\varepsilon^{Q}).
\end{aligned}
\end{equation*}
Next, for $Q \geq 4$, we compute
\begin{equation*}
\begin{aligned}
\int_{\Omega}u_\ep^{2}\mathrm{d} z&=\int_{\Omega}(\eta U_{\varepsilon})^{2}\mathrm{d} z \leq \int_{B_{2R}(0)} U_{\varepsilon}^{2}\mathrm{d} z=\int_{B_{2R}(0)} \varepsilon^{2-Q} U^{2}\big(\frac{1}{\varepsilon} x, \frac{1}{\varepsilon^{\alpha +1}} y\big)\mathrm{d} z \\
&=\int_{B_{\frac{2R}{\varepsilon}}(0)} \varepsilon^{2} U^{2}\mathrm{d} z =\varepsilon^{2}\Big(\int_{B_{1}(0)} U^{2}\mathrm{d} z+\int_{B \frac{2R}{\varepsilon}(0) \backslash B_{1}(0)} U^{2}\mathrm{d} z\Big)\\
&\leq \varepsilon^{2}\Big(C+\int_{1<d_\alpha(z)<\frac{2R}{\ep}}\frac{C}{d_\alpha(z)^{2 Q-4}} \mathrm{d} z \Big) = \varepsilon^{2}\Big(C+C\int_{1}^{\frac{2R}{\varepsilon}} r^{3-Q} \mathrm{d} r \Big) \\
&=\left\{\begin{array}{ll}
C \varepsilon^{2}|\log \varepsilon|+O(\varepsilon^{2}), & \text { if } ~Q=4, \\
C \varepsilon^{2}+O(\varepsilon^{Q-2}), \quad &\text { if }~Q>4 .
\end{array}\right.
\end{aligned}
\end{equation*}
Especially, for $2<Q<4$, we obtain
\begin{equation*}
\begin{aligned}
\int_{\Omega}u_\ep^{2}\mathrm{d} z & \leq \int_{B_{2R}(0)} \varepsilon^{2-Q} U^{2}\big(\frac{1}{\varepsilon} x, \frac{1}{\varepsilon^{\alpha +1}} y\big)\mathrm{d} z \leq C\ep^{2-Q} \int_{0<d_\alpha(z)<2R}\frac{1}{(\ep^{-1}d_\alpha(z))^{2(Q-2)}}\mathrm{d} z\\
&=C\ep^{Q-2}\int_{0}^{2R}r^{3-Q}\mathrm{d} r=C\ep^{Q-2}.
\end{aligned}
\end{equation*}
Hence, we have
\begin{equation*}
\begin{aligned}
\int_{\Omega}u_\ep^{2}\mathrm{d} z  \leq \left\{\begin{array}{ll}
C \varepsilon^{Q-2}, & \text { if } ~2<Q<4, \\
C \varepsilon^{2}|\log \varepsilon|+O(\varepsilon^{2}), & \text { if } ~Q=4, \\
C\varepsilon^{2}+O(\varepsilon^{Q-2}), & \text { if }~ Q>4.
\end{array}\right.
\end{aligned}
\end{equation*}
Actually, we can also give the lower bound of $\int_{\Omega}u_\ep^{2}\mathrm{d} z$ by using $U(z) \geq \frac{C}{d_\alpha(z)^{Q-2}+1}$. Since this result is not required for the subsequent proof, we omit its detailed derivation. Then we complete the proof.
\end{proof}
It is worth noting that the similar estimates as in Lemma \ref{gt} hold on the Carnot group (cf. \cite{Loiudice2007}).
\section{$\lambda>0$: Existence of non-negative solutions}\label{S3}
In this section, we first give a significant estimate of the logarithmic term for $Q\geq 4$. Building on this result, we need to demonstrate that the mountain pass value $c_\lambda$ of $I_+$ is strictly less than the non-compactness level $\frac{1}{Q}c_{m,n,\alpha}^Q$ and finish the proof of Theorem \ref{thm1.1}.
\vspace{-0.5cm}
\subsection{$Q>4$:  Estimates of logarithmic term}\label{s3.1}
\vs{2mm}
\begin{Lem}\label{lem1}
    For $u_{\ep}$ defined by $(\ref{a1.6})$, if $\ep>0$ is sufficiently small, then for $Q>4$, we have
\begin{equation}\label{a602}
  \int_{\Omega}u_\ep^{2} \log u_\ep^{2}\mathrm{d} z=(Q-2)
\ep^{2} |\log \ep|\int_{\mathbb{R}^N} U^2 \mathrm{d}z +O(\ep^{2}).
\end{equation}
\end{Lem}

\begin{proof}
Recalling that the homogeneous metric  $d_\alpha(z)=(|x|^{2(\alpha+1)}+|y|^{2})^{\frac{1}{2(\alpha+1)}}$. The circumstance in which $Q>4$ is comparatively uncomplicated and independent of the particular form of $U(x,y)$. And our proof requires only the asymptotic estimate provided by $(\ref{a1.3})$.

Above all, we split $\log u_{\ep}^{2}$ into $\log \eta^{2}+\log U_{\varepsilon}^{2}$ as follows.
\begin{equation*}
    \begin{aligned}
\int_{\Omega} u_{\ep}^{2} \log u_{\ep}^{2} \mathrm{d} z =\int_{\Omega} \eta^{2} U_{\varepsilon}^{2}\log \eta^{2}  \mathrm{d} z+\int_{\Omega} \eta^{2} U_{\varepsilon}^{2} \log U_{\varepsilon}^{2} \mathrm{d} z=:B_{1}+B_{2}.
\end{aligned}
\end{equation*}
Taking into account that $\log\eta=0$ in $B_R(0)$ and $|t^2 \log t^2| \leq C$ for $0 \leq t \leq 1$, we get
\begin{equation}\label{a58}
\begin{aligned}
|B_{1}|&=\Big|\int_{B_{2 R}(0) \backslash B_{R}(0)} \eta^{2}\log \eta^{2} U_\ep^2\mathrm{d} z\Big| \leq C\int_{B_{2 R}(0) \backslash B_{R}(0)}U_\ep^2\mathrm{d} z\\
&=C\int_{B_{2 R}(0) \backslash B_{R}(0)}\varepsilon^{2-Q} U^{2}\big(\frac{1}{\varepsilon} x, \frac{1}{\varepsilon^{\alpha+1}} y\big)\mathrm{d} z\\
&=C\int_{B_{\frac{2 R}{\ep}}(0) \backslash B_{\frac{R}{\ep}}(0)}\ep^2U^2\mathrm{d} z \leq C \int_{\frac{R}{\ep}< d_\alpha(z)< \frac{2R}{\ep}} \ep^2 \frac{1}{d_\alpha(z)^{2Q-4}}\mathrm{d} z\\
&=C\int_{\frac{R}{\ep}}^{\frac{2 R}{\ep}}\ep^2\frac{r^{Q-1}}{r^{2Q-4}}\mathrm{d}r=C\ep^{Q-2},
\end{aligned}
\end{equation}
which means that
\begin{equation}\label{a101}
    B_1=O(\ep^{Q-2}).
\end{equation}
As for $B_2$, we have
\begin{equation*}
\begin{aligned}
B_{2} &=\int_{\Omega \backslash B_{R}(0)} \eta^{2} U_{\varepsilon}^{2} \log U_{\varepsilon}^{2} \mathrm{d} z + \int_{B_{R}(0)} U_{\varepsilon}^{2} \log U_{\varepsilon}^{2} \mathrm{d} z =:B_{3} + B_{4}.
\end{aligned}
\end{equation*}
For $\delta \in (0,1)$, using $(\ref{b1.2})$, we receive
\begin{equation}\label{a0.2}
\begin{aligned}
|B_{3}| &\leq \int_{\Omega \backslash B_{R}(0)}|U_{\varepsilon}^{2} \log U_{\varepsilon}^{2}| \mathrm{d} z \leq C_{\delta} \int_{\Omega \backslash B_{R}(0)} ( U_{\varepsilon}^{2-\delta}+U_{\varepsilon}^{2+\delta} ) \mathrm{d} z \\
&=C_{\delta}\int_{\Omega \backslash B_{R}(0)} \varepsilon^{\frac{(2-Q)(2-\delta)}{2}} U^{2-\delta}(\frac{1}{\varepsilon} x, \frac{1}{\varepsilon^{\alpha+1}} y) \mathrm{d} z+C_{\delta}\int_{\Omega \backslash B_{R}(0)} \varepsilon^{\frac{(2-Q)(2+\delta)}{2}} U^{2+\delta}(\frac{1}{\varepsilon} x, \frac{1}{\varepsilon^{\alpha+1}}y) \mathrm{d} z \\
&\leq
C_{\delta}\int_{\mathbb{R}^N \backslash B_{\frac{R}{\varepsilon}}(0)} \varepsilon^{Q+\frac{(2-Q)(2-\delta)}{2}} U^{2-\delta} \mathrm{d} z+C_{\delta}\int_{\mathbb{R}^N \backslash B_{\frac{R}{\varepsilon}}(0)} \varepsilon^{Q+\frac{(2-Q)(2+\delta)}{2}} U^{2+\delta} \mathrm{d} z.
\end{aligned}
\end{equation}
We may as well take $\delta>0$ small enough, then we have,
\begin{equation*}
    (Q-2)(2+\delta) > (Q-2)(2-\delta)> Q>4.
\end{equation*}
Thus,
\begin{equation}
\begin{aligned}
|B_3|&\leq C \int_{\frac{R}{\varepsilon}}^{+\infty} \varepsilon^{Q+\frac{(2-Q)(2-\delta)}{2}} \frac{r^{Q-1}}{r^{(Q-2)(2-\delta)}} \mathrm{d} r+ C\int_{\frac{R}{\varepsilon}}^{+\infty} \varepsilon^{Q+\frac{(2-Q)(2+\delta)}{2}} \frac{r^{Q-1}}{r^{(Q-2)(2+\delta)}} \mathrm{d} r \\
&=C\varepsilon^{Q+\frac{(2-Q)(2-\delta)}{2}} r^{Q-(Q-2)(2-\delta)}\big|_{\frac{R}{\varepsilon}} ^{+\infty}  +C \varepsilon^{Q+\frac{(2-Q)(2+\delta)}{2}} r^{Q-(Q-2)(2+\delta)}\big|_{\frac{R}{\varepsilon}} ^{+\infty}\\
&=C \varepsilon^{\frac{(Q-2)(2-\delta)}{2}}+C\varepsilon^{\frac{(Q-2)(2+\delta)}{2}}
\leq O(\ep^2).
\end{aligned}
\end{equation}
 Next, we compute
\begin{equation*}
\begin{aligned}
B_{4}&=\int_{B_{R}(0)} U_{\varepsilon}^{2} \log U_{\varepsilon}^{2} \mathrm{d} z = \int_{B_{R}(0)} \varepsilon^{2-Q} U^{2}(\frac{1}{\varepsilon} x, \frac{1}{\varepsilon^{\alpha+1}} y) \log \big(\varepsilon^{2-Q} U^{2}(\frac{1}{\varepsilon} x, \frac{1}{\varepsilon^{\alpha+1}} y)\big) \mathrm{d}z \\
&=\int_{B_{R}(0)} \varepsilon^{2-Q}U^{2}(\frac{1}{\varepsilon} x, \frac{1}{\varepsilon^{\alpha+1}} y) \log \varepsilon^{2-Q}  \mathrm{d}z \\
&\ \ \ \ +\int_{B_{R}(0)} \varepsilon^{2-Q} U^{2}(\frac{1}{\varepsilon} x, \frac{1}{\varepsilon^{\alpha+1}} y) \log \big(U^{2}(\frac{1}{\varepsilon} x, \frac{1}{\varepsilon^{\alpha+1}} y)\big) \mathrm{d}z\\
&=\int_{B_{ \frac{R}{\varepsilon}}(0)} \varepsilon^{2} \log \varepsilon^{2-Q} U^{2} \mathrm{d}z+\int_{B _{\frac{R}{\varepsilon}} (0)} \varepsilon^{2} U^{2}\log U^{2} \mathrm{d}z. \\
\end{aligned}
\end{equation*}
By \eqref{a1.3}, for any constant $\varrho$ with $(Q-2)(2-\varrho)> Q$,
$$\int_{\mathbb{R}^{N} \backslash B_{1}(0)} U^{2-\varrho} \mathrm{d}z\leq \int_{\mathbb{R}^{N} \backslash B_{1}(0)} \frac{C}{d_\alpha(z)^{(Q-2)(2-\varrho)}} \mathrm{d}z \leq C.$$
Since $U$ is bounded in $B_{1}(0)$,  we conclude $U(x,y) \in L^{2-\varrho}(\mathbb{R}^N)$. Therefore,
\begin{equation*}
\begin{aligned}
B_4&\geq (Q-2)\varepsilon^{2}|\log \varepsilon|\int_{\mathbb{R}^N} U^2 \mathrm{d}z +\varepsilon^{2} \int_{B_{\frac{R}{\varepsilon}}(0) \cap \{U <1\}} U^{2} \log U^{2} \mathrm{d}z
\\
&\geq (Q-2) \varepsilon^{2}|\log \varepsilon|\int_{\mathbb{R}^N} U^2 \mathrm{d}z -C \varepsilon^{2} \int_{B_{\frac{R}{\varepsilon}} (0)} U^{2-\delta} \mathrm{d}z \\
&=  (Q-2)\varepsilon^{2}|\log \varepsilon|\int_{\mathbb{R}^N} U^2 \mathrm{d} z+O(\ep^2).
\end{aligned}
\end{equation*}
On the other hand, one can also get
\begin{equation}\label{a0.1}
    B_4 \leq (Q-2)\varepsilon^2|\log\varepsilon|\int_{\mathbb{R}^N} U^2 \mathrm{d}z+O(\varepsilon^2).
\end{equation}
Combined with $(\ref{a101})$, $(\ref{a0.2})$--$(\ref{a0.1})$, we get
\begin{equation*}
    \int_{\Omega} u_{\varepsilon}^{2} \log u_{\varepsilon}^{2} \mathrm{d} z= (Q-2) \varepsilon^{2}|\log \varepsilon|\int_{\mathbb{R}^N} U^2 \mathrm{d} z+O(\varepsilon^{2}),
\end{equation*}
which completes the proof.
\end{proof}

\vs{2mm}
\subsection{$Q=4$: Estimates of logarithmic term}\label{s3.2}
\
\newline

Before getting the estimates, we need the following useful lemma:
\begin{Lem}\label{lem11}
Suppose that $\psi(r)>0$ and $\psi \in C^2[0,\frac{1}{2})\cap C^0[0,\frac{1}{2}]$. Besides, if for some $\theta \in [0,1),$ we have $\psi$ satisfies
\begin{equation}\label{a50}
  \lim_{r \to (\frac{1}{2})^-}(\frac{1}{2}-r)^{\theta}\psi'(r)=0.
\end{equation}
Then $\psi$ is $(1-\theta)$-H\"{o}lder continuous.
\end{Lem}
\begin{proof}
Consider $f(r):=\psi(\frac{1}{2}-r^{\frac{1}{1-\theta}})$, then we have
$$\lim_{r \to 0^+}f'(r)=\lim_{r \to 0^+}\frac{1}{\theta-1}r^{\frac{\theta}{1-\theta}}\psi'(\frac{1}{2}-r^{\frac{1}{1-\theta}})=\lim_{s \to (\frac{1}{2})^-}\frac{1}{\theta-1}(\frac{1}{2}-s)^{\theta}\psi'(s)=0.$$
Therefore, $f\in C^1[0,\frac{1}{2}]$ is Lipschitz continuous, so we obtain, for $r_1,r_2\in [0,\frac{1}{2}]$,
$$|\psi(r_1)-\psi(r_2)|=|f((\frac
{1}{2}-r_1)^{1-\theta})-f((\frac
{1}{2}-r_2)^{1-\theta})|\leq C|(\frac
{1}{2}-r_1)^{1-\theta}-(\frac
{1}{2}-r_2)^{1-\theta}|\leq C|r_1-r_2|^{1-\theta} ,$$
which means that $\psi$ is $(1-\theta)$-H\"{o}lder continuous.
\end{proof}
\vs{2mm}
Now, we give the main Proposition of this section:
\begin{Prop}\label{lem100}
    Let $Q=4$, for a fixed small $R$ and $u_{\ep}$ is defined by $(\ref{a1.6})$.
    When $R\geq \ep \to 0^+$,
     we have
\begin{equation}\label{a600}
\int_{\Omega} u_{\varepsilon}^{2} \log u_{\varepsilon}^{2}\mathrm{d} z \geq C(R)\ep^2|\log\ep|-C_{3}\ep^2|\log\ep|,
\end{equation}
where $C_{3}$ is a constant independent of $\ep$, $R$ and $C(R)\to +\infty$ as $R \to 0^{+}$.

\end{Prop}
\begin{proof}
Due to the inherent complexity of the logarithmic terms, direct utilization of the asymptotic estimate of the solution $U(x)$ at infinity, as we did in Lemma \ref{gt}, is not feasible for completing the proof. To overcome this obstacle, we must deduce precise estimates for the coefficients of the term $\ep^2|\log \ep|^2$, which is a difficult point for us.
Firstly, we point out that by Lemma \ref{gt}, in the case $Q=4$ we have
$\int_{\Omega} u_{\ep}^2 \mathrm{d} z \leq C\ep^2|\log\ep|+O(\ep^2),$ thus  estimate \eqref{a600} remains invariant under scaling and multiplication by a constant. Therefore, we only need to prove the estimate \eqref{a600} for $u_{\ep}(x,y):=\eta(x,y)U_{\ep}(x,y)=\ep^{\frac{2-Q}{2}}\eta(x,y)U(\frac{1}{\ep} x, \frac{1}{\ep^{\alpha +1}}y)$
under the following three cases
 of $Q=m+n(\alpha +1)=4$ and $m, n \in \mathbb{N}^+$ even though \eqref{Scale} maybe fails:
\begin{enumerate}
  \item $m=2, n=1, \alpha=1,$ $$U(x,y)=\frac{1}{\big((|x|^2+1)^2+|y|^2\big)^{\frac{1}{2}}};$$

  \item $m=1$, $n=1$, $\alpha=2$,
  \[ U(x,y)=\frac{1}{\big((|x|^{3}+1)^2+|y|^2\big)^{\frac{1}{3}}} \psi\Big(\big|\frac{(|x|^{3}+1,y)}{(|x|^{3}+1)^2+|y|^2} -(\frac{1}{2},0)\big|\Big);\]

  \item  $m=1$, $n=2$, $\alpha=\frac{1}{2}$,

  \[U(x,y)=\frac{1}{\big((|x|^{\frac{3}{2}}+1)^2 +|y|^2\big)^{\frac{2}{3}}}\psi\Big(\big|\frac{(|x|^{\frac{3}{2}} +1,y)}{(|x|^{\frac{3}{2}}+1)^2+|y|^2} -(\frac{1}{2},0)\big|\Big).\]
\end{enumerate}
 Recall $$\int_{B_{r_2}(0)\backslash B_{r_1}(0)} f(d_\alpha(z))\mathrm{d} z=Q|B_1(0)|\int_{r_1}^{r_2}f(\rho)\rho^{Q-1}\mathrm{d}\rho.$$

Specially, we point out that since $\ep, R$ are  sufficiently  small, we have $\log \ep<0$ and $\log ((|x|^{\alpha+1} +\ep^{\alpha+1})^2+|y|^2)<0$ in $B_{2R}(0).$

\vspace{0.3cm}
{\bf{Case 1: The case for $Q=4, m=2, n=1, \alpha=1$.}}\vspace{0.15cm}

\indent In this situation,
$U_{\ep}(x,y)$ can be expressed as follows.
$$U_{\ep}(x,y)=\frac{\ep}{\big((|x|^2+\ep^2)^2+|y|^2\big)^{\frac{1}{2}}}.$$
Now, we are required to estimate the lower bound of $\int_{\Omega} u_{\varepsilon}^{2} \log u_{\varepsilon}^{2} \mathrm{d} z$ under the homogeneous metric $d_\alpha(z)$ defined as $(|x|^4+|y|^2)^{\frac{1}{4}}$.
$$\int_{\Omega} u_{\varepsilon}^{2} \log u_{\varepsilon}^{2} \mathrm{d} z =\int_{\Omega} U_{\varepsilon}^{2}\eta^{2} \log \eta^{2} \mathrm{d} z+\int_{\Omega} \eta^{2} U_{\varepsilon}^{2} \log U_{\varepsilon}^{2} \mathrm{d} z=:D_{1}+D_{2}.$$
Similar as $(\ref{a58})$, we obtain
\begin{equation}\label{a96}
\begin{aligned}
|D_{1}|&=\Big|\int_{B_{2 R}(0) \backslash B_{R}(0)}\frac{\varepsilon^{2}}{(|x|^{2}+\ep^{2})^{2}+|y|^{2}} \eta^{2} \log \eta^{2} \mathrm{d} z\Big| \leq C \int_{B_{2 R}(0) \backslash B_{R}(0)} \frac{\varepsilon^{2}}{(|x|^{2}+\ep^{2})^{2}+|y|^{2}}\mathrm{d} z\\
&\leq C \varepsilon^{2} \int_{B_{2R}(0)\backslash B_{R}(0)} \frac{1}{(|x|^{4}+|y|^2)+\varepsilon^{4}}\mathrm{d} z =C\varepsilon^{2} \int_{R}^{2 R} \frac{r^{3}}{r^{4}+\varepsilon^{4}} \mathrm{d} r\\
&=C \varepsilon^{2}\big(\frac{1}{4} \log (r^{4}+\varepsilon^{4})\big)\big|_{R} ^{2 R}=O(\varepsilon^{2}).
\end{aligned}
\end{equation}
 As for $D_2$, we have
\begin{equation*}
\begin{aligned}
D_{2}&=2\varepsilon^{2} \log \varepsilon \int_{\Omega} \frac{\eta^{2}}{(|x|^{2}+\varepsilon^{2})^{2}+|y|^{2}}\mathrm{d} z-\varepsilon^{2} \int_{\Omega} \frac{\eta^{2} \log \big((|x|^{2}+\varepsilon^{2})^{2}+|y|^{2}\big)}{(|x|^{2}+\varepsilon^{2})^{2} +|y|^{2}} \mathrm{d} z=: D_3-D_4.
\end{aligned}
\end{equation*}
Estimating the lower bound of $D_2$ is equivalent to estimating the lower bound of $D_3$ and the upper bound of $D_4$. We obtain
\begin{equation}\label{a95}
\begin{aligned}
D_3& \geq 2\varepsilon^{2} \log \varepsilon\int_{B_{2 R}(0)} \frac{1}{(|x|^{4}+|y|^{2})+\varepsilon^{4}}\mathrm{d} z =2Q|B_1(0)|\varepsilon^{2} \log \varepsilon \int_{0}^{2 R} \frac{r^{3}}{r^{4}+\varepsilon^{4}}\mathrm{d} r\\
&=2Q|B_1(0)|\ep^{2} \log \ep\big(\frac{1}{4} \log(r^{4}+\varepsilon^{4})\big)\big|_{0} ^{2 R}\\
&=\big(\frac{1}{2}Q|B_1(0)|\log(16 R^{4}+\ep^{4})\big)\ep^{2}\log\ep-2Q|B_1(0)|\ep^2\log^2\ep
\end{aligned}
\end{equation}
and
$D_4=:D_5-D_6,$
where
$$D_5=\ep^2\int_{\Omega} \frac{\eta^{2} \log (|x|^{4}+|y|^{2}+\varepsilon^{4})}{|x|^{4}+|y|^{2}+\varepsilon^{4}}\mathrm{d} z,$$
$$D_6=\ep^2\int_{\Omega} \Big(\frac{\eta^{2} \log (|x|^{4}+|y|^{2}+\varepsilon^{4})}{|x|^{4}+|y|^{2}+\varepsilon^{4}}-\frac{\eta^{2} \log \big((|x|^{2}+\varepsilon^{2})^{2}+|y|^{2}\big)}{(|x|^{2}+\varepsilon^{2})^{2}+|y|^{2}} \Big)\mathrm{d} z.$$
Then by integration, we have
\begin{equation}\label{a94}
\begin{aligned}
D_{5} &\leq \ep^2\int_{B_R(0)} \frac{\log((|x|^{4}+|y|^{2})+\varepsilon^{4})}{(|x|^{4}+|y|^{2}) +\varepsilon^{4}}\mathrm{d} z=Q|B_1(0)|\ep^2 \int_{0}^{R} \frac{\log(r^{4}+\varepsilon^{4})}{r^{4}+\varepsilon^{4}} r^{3} \mathrm{d}r \\
&=Q|B_1(0)| \ep^2\big(\frac{1}{8}\log ^{2}(r^{4}+\varepsilon^{4})\big)\big|_{0} ^{R}\\
&=\big(\frac{1}{8} Q|B_1(0)| \log ^{2}(R^{4}+\varepsilon^{4})\big)\ep^2-2 Q|B_1(0)|\ep^2 \log^{2}\varepsilon.
\end{aligned}
\end{equation}
The lower bound estimate for $D_{6}$ is relatively complex and requires adding and subtracting terms, which can be seen as
\begin{equation*}
\begin{aligned}
D_{6}=\ep^2\int_{\Omega} &\eta^{2}\Big(\frac{\log(|x|^{4}+|y|^{2}+\varepsilon^{4})} {|x|^{4}+|y|^{2}+\varepsilon^{4}} -\frac{\log(( |x|^{2}+\varepsilon^{2})^{2}+|y|^{2})}{(|x|^{2}+\varepsilon^{2})^{2} +|y|^{2}}\Big)\mathrm{d} z \\
=\ep^2\int_{\Omega} &\eta^{2}\Big(\frac{\log(|x|^{4}+|y|^{2}+\varepsilon^{4})} {|x|^{4}+|y|^{2}+\varepsilon^{4}} -\frac{\log\big((|x|^{2}+\varepsilon^{2})^{2}+|y|^{2}\big)} {|x|^{4}+|y|^{2}+\varepsilon^{4}} +\frac{\log\big((|x|^{2}+\varepsilon^{2})^{2}+|y|^{2}\big)} {|x|^{4}+|y|^{2}+\varepsilon^{4}}\\
&-\frac{\log \big((|x|^{2}+\varepsilon^{2})^{2}+|y|^{2}\big)}{(\varepsilon^{2}+|x|^2)^{2}+|y|^{2}}\Big)\mathrm{d} z \\
=\ep^2\int_{\Omega} &\eta^{2}\bigg(\frac{\log \frac{|x|^{4}+|y|^{2}+\varepsilon^{4}}{(|x|^{2}+\varepsilon^{2})^{2}+|y|^{2}}} {|x|^{4}+|y|^{2}+\varepsilon^{4}}+\frac{2 \varepsilon^{2}|x|^{2}\log \big((|x|^{2}+\varepsilon^{2})^{2}+|y|^{2}\big) }{(|x|^{4}+|y|^{2}+\varepsilon^{4})\big((|x|^{2}+\varepsilon^{2})^{2}+|y|^{2}\big)}\bigg) \mathrm{d} z=:D_7+D_8.
\end{aligned}
\end{equation*}
Since $\frac{1}{2}\leq \frac{|x|^{4}+|y|^{2}+\varepsilon^{4}}{(|x|^{2}+\varepsilon^{2})^{2}+|y|^2}\leq 1$, we get
\begin{equation}\label{a93}
\begin{aligned}
D_{7} &\geq \ep^2\int_{B_{2 R}(0)} \frac{\log \frac{|x|^{4}+|y|^{2}+\varepsilon^{4}}{(|x|^{2}+\varepsilon^{2})^{2}+|y|^2}} {|x|^{4}+|y|^{2}+\varepsilon^{4}}\mathrm{d} z \geq\big(\log\frac{1}{2}\big)\ep^2 \int_{B_{2 R}(0)}\frac{1}{(|x|^{4}+|y|^{2})+\varepsilon^{4}}\mathrm{d} z\\
&=\big(Q|B_1(0)|\log \frac{1}{2}\big)\ep^2\int_{0}^{2 R} \frac{r^{3}}{r^{4}+\varepsilon^{4}} \mathrm{d} r= \big(Q|B_1(0)|\log \frac{1}{2}\big)\ep^2\big(\frac{1}{4} \log(r^{4}+\varepsilon^{4})\big)\big|_{0} ^{2 R} \\
&=\big(\frac{1}{4}Q|B_1(0)|\log \frac{1}{2}\log(16 R^{4}+\varepsilon^{4})\big)\ep^2-\big(Q|B_1(0)|\log\frac{1}{2}\big) \ep^2\log\varepsilon
\end{aligned}
\end{equation}
and
\begin{equation}\label{a92}
\begin{aligned}
D_{8} & \geq\ep^2\int_{B_{2 R}} \log\big((|x|^{2}+\varepsilon^{2})^{2}+|y|^{2}\big) \frac{2 \varepsilon^{2}|x|^{2}}{(|x|^{4}+|y|^{2}+\varepsilon^{4})\big( (|x|^{2}+\varepsilon^{2})^{2}+|y|^{2}\big)}\mathrm{d} z \\
& \geq \ep^2\int_{B_{2 R}} \log ((|x|^{4}+|y|^{2})+\varepsilon^{4}) \frac{2 \varepsilon^{2}(|x|^4+|y|^2)^{\frac{1}{2}}}{((|x|^{4}+|y|^{2})+\varepsilon^{4})^{2}} \mathrm{d} z\\
& =2 Q|B_1(0)|\ep^2 \int_{0}^{2 R} \frac{\varepsilon^{2} r^{5} \log (r^{4}+\varepsilon^{4})}{(r^{4}+\varepsilon^{4})^{2}} \mathrm{d} r
=2 Q|B_1(0)|\ep^2\int_{0}^{\frac{2 R}{\varepsilon}}\frac{r^5\log\big(\ep^4(r^{4}+1)\big)}{(r^{4}+1)^{2}}\mathrm{d} r\\
&=8 Q|B_1(0)|\ep^2\log\ep\int_{0}^{\frac{2 R}{\varepsilon}}\frac{r^5}{(r^{4}+1)^{2}}\mathrm{d} r+2 Q|B_1(0)| \ep^2\int_{0}^{\frac{2 R}{\varepsilon}}\frac{r^5\log(r^{4}+1)}{(r^{4}+1)^{2}}\mathrm{d} r\\
&\geq 2Q|B_1(0)|\ep^2\log\ep\big(\arctan(r^2)-\frac{r^2}{r^4+1}\big)\big|_0^{\frac{2 R}{\varepsilon}}+O(\ep^2)\\
&=2Q|B_1(0)|\arctan\big(\frac{4R^2}{\ep^2}\big)\ep^2\log\ep- \frac{8Q|B_1(0)|R^2}{16R^4+\ep^4}\ep^4\log\ep
+O(\ep^2).
\end{aligned}
\end{equation}
Hence we can obtain from $(\ref{a96})-(\ref{a92})$ that
\begin{equation*}
\begin{aligned}
\int_{\Omega} u_{\varepsilon}^{2} \log u_{\varepsilon}^{2}\mathrm{d} z &\geq O(\ep^2)+\big(\frac{1}{2}Q|B_1(0)| \log(16R^{4}+\ep^4)\big)\ep^2\log\ep-2Q|B_1(0)|\ep^2\log^2\ep\\
&\ \ \ \ - \big(\frac{1}{8}Q|B_1(0)|\log^2(R^{4}+\ep^4)\big)\ep^2 +2Q|B_1(0)|\ep^2\log^2\ep\\
&\ \ \ \ +\big( \frac{1}{4}Q|B_1(0)|\log \frac{1}{2}\log(16 R^{4}+\varepsilon^{4})\big)\ep^2-\big(Q|B_1(0)|\log \frac{1}{2}\big)\ep^2\log\ep\\
&\ \ \ \ +2Q|B_1(0)|\arctan\big(\frac{4R^2}{\ep^2}\big)\ep^2\log\ep -\frac{8Q|B_1(0)|R^2}{16R^4+\ep^4}\ep^4\log\ep\\
&=\big(-Q|B_1(0)|\log2-2Q|B_1(0)|\arctan\big(\frac{4R^2}{\ep^2}\big) -\frac{1}{2}Q|B_1(0)| \log(16R^{4}+\ep^4)\big)\ep^2|\log\ep|\\
&\ \ \ \ +O(\ep^2)\\
& \geq Q|B_1(0)| \big(-\log2-\pi - \frac{1}{2}\log(17R^4) \big) \ep^2|\log\ep| + O(\ep^2).
\end{aligned}
\end{equation*}

\vspace{0.3cm}
{\bf{Case 2: The case for $Q=4$, $m=1$, $n=1$, $\alpha=2$.}}\vspace{0.15cm}
\
\newline
In this situation,
$$U_\ep(x,y)=\frac{\ep}{\big((|x|^3+\ep^3)^2+|y|^2\big)^{\frac{1}{3}}} \psi\Big(\big|\frac{\ep^3(|x|^3+\ep^3,y)}{(|x|^3+\ep^3)^2+|y|^2} -(\frac{1}{2},0)\big|\Big),$$
where $\psi(r)>0$, $\psi(\frac{1}{2})=1$, $\psi'(0)=0$, $\psi \in C^2[0,\frac{1}{2})\cap C^0[0,\frac{1}{2}]$ satisfies
\begin{equation}\label{a52}
  \lim_{r \to (\frac{1}{2})^-}(\frac{1}{4}-r^2)^{\frac{2}{3}}\psi'(r)=0.
\end{equation}

According to Lemma $\ref{lem11}$ and $(\ref{a52})$, $\psi$ is $\frac{1}{3}$-H\"older continuous, which means that
\begin{equation}\label{a53}
  \begin{aligned}
  \Big|\psi^2\big(|\frac{\ep^3(|x|^3+\ep^3,y)}{(|x|^3+\ep^3)^2+|y|^2}-(\frac{1}{2},0)|\big)-1\Big|&=C\Big|\psi\big(|\frac{\ep^3(|x|^3+\ep^3,y)}{(|x|^3+\ep^3)^2+|y|^2}-(\frac{1}{2},0)|\big)-\psi\big(|(\frac{1}{2},0)|\big)\Big|\\
  &\leq C\Big|\big|\frac{\ep^3(|x|^3+\ep^3,y)}{(|x|^3+\ep^3)^2+|y|^2}- (\frac{1}{2},0)\big|-\big|(\frac{1}{2},0)\big|\Big|^{\frac{1}{3}}\\
  &\leq C\big|\frac{\ep^3(|x|^3+\ep^3,y)}{(|x|^3+\ep^3)^2+|y|^2}\big|^{\frac{1}{3}}\\
  &=C\frac{\ep}{\big((|x|^3+\ep^3)^2+|y|^2\big)^{\frac{1}{6}}}.
  \end{aligned}
\end{equation}

Next, we present the lower bound estimates of $\int_{\Omega} u_{\varepsilon}^{2} \log u_{\varepsilon}^{2} \mathrm{d} z$ under the homogeneous norm defined as $d_{\alpha}(z)=(|x|^6+|y|^2)^{\frac{1}{6}}$. The initial splitting method is the same as in Case 1.
\begin{equation}\label{a000}
  \int_{\Omega} u_{\varepsilon}^{2} \log u_{\varepsilon}^{2} \mathrm{d} z =\int_{\Omega} \eta^{2} \log \eta^{2} U_{\varepsilon}^{2} \mathrm{d} z+\int_{\Omega} \eta^{2} U_{\varepsilon}^{2} \log U_{\varepsilon}^{2} \mathrm{d} z=:E_{1}+E_{2}.
\end{equation}
Recall the following basic inequalities:
\begin{equation}\label{a60}
  C(a^k+b^k)\leq (a+b)^k \leq a^k+b^k, \mbox{ for } a, b\geq 0, \mbox{ and }k \in (0,1).
\end{equation}
And for $a, b\geq 0$, $k \geq 1$, one has
\begin{equation}\label{a61}
  a^k+b^k\leq (a+b)^k \leq C_1(a^k+b^k),
\end{equation}
Thus we obtain by $(\ref{a60})$ and $(\ref{a61})$ that
\begin{equation}\label{a62}
  \begin{aligned}
    |E_{1}|&=\Big|\int_{B_{2 R}(0) \backslash B_{R}(0)} \eta^{2} \log \eta^{2} \frac{\ep^{2}\psi^2}{\big((|x|^3+\ep^3)^2+|y|^2\big)^{\frac{2}{3}}}\mathrm{d} z\Big|\\
    &\leq C \int_{B_{2 R}(0) \backslash B_{R}(0)} \frac{\ep^{2}}{((|x|^6+|y|^2)+\ep^6)^{\frac{2}{3}}}\mathrm{d} z=C\ep^{2} \int_{R}^{2 R} \frac{r^{3}}{(r^6+\ep^6)^{\frac{2}{3}}} \mathrm{d} r\\
    &\leq C\ep^{2} \int_{R}^{2 R} \frac{r^{3}}{r^{4}+\ep^{4}} \mathrm{d} r\leq C\ep^{2},
  \end{aligned}
\end{equation}
which means that $E_1=O(\varepsilon^{2})$. After that, we compute $E_2$.
\begin{equation*}
  \begin{aligned}
   E_2 & =\int_{\Omega} \eta^{2} \psi^{2} \ep^{2} \frac{1}{\big((|x|^3+\ep^3)^2+|y|^2\big)^{\frac{2}{3}}} \log \ep^{2}\mathrm{d} z+\int_{\Omega} \eta^{2} \psi^{2} \varepsilon^{2} \frac{1}{\big((|x|^3+\ep^3)^2+|y|^2\big)^{\frac{2}{3}}} \log \psi^{2} \mathrm{d} z\\
   &\ \ \ \ -\int_{\Omega} \eta^{2} \psi^{2} \varepsilon^{2} \frac{1}{\big((|x|^3+\ep^3)^2+|y|^2\big)^{\frac{2}{3}}} \log \big((|x|^3+\ep^3)^2+|y|^2\big)^{\frac{2}{3}}\mathrm{d} z\\
   &=:E_{3}+E_{4}-E_{5}.
   \end{aligned}
\end{equation*}
While estimating $E_{3}$, $E_{4}$ and $E_{5}$, we consistently decompose $\psi^2$ into $(\psi^2-1)+1$ to leverage the $\frac{1}{3}$-H\"older continuity of $\psi$ in the process of proof. Then $E_3$ can be written as
\begin{equation*}
  \begin{aligned}
  E_3=\int_{\Omega} \frac{\eta^{2} \ep^{2}\log \ep^{2}}{\big((|x|^3+\ep^3)^2+|y|^2\big)^{\frac{2}{3}}}\mathrm{d} z+\int_{\Omega} \frac{\eta^{2}(\psi^{2}-1)\ep^{2}\log \ep^{2}}{\big((|x|^3+\ep^3)^2+|y|^2\big)^{\frac{2}{3}}}\mathrm{d} z=:E_6+E_7.
  \end{aligned}
\end{equation*}
For $E_6$, we have
\begin{equation}\label{a63}
\begin{aligned}
E_6 &\geq \ep^2\log\ep^2 \int_{B_{2R}(0)}\frac{1}{\big((|x|^3+\ep^3)^2+|y|^2\big)^{\frac{2}{3}}} \mathrm{d} z\geq \ep^2\log\ep^2 \int_{B_{2R}(0)}\frac{1}{((|x|^6+|y|^2)+\ep^6)^{\frac{2}{3}}}\mathrm{d} z\\
&=2Q|B_1(0)|\ep^2\log\ep \int_{0}^{2 R}\frac{r^3}{(r^6+\ep^6)^{\frac{2}{3}}} \mathrm{d} r=2Q|B_1(0)|\ep^2\log\ep \int_{0}^{\frac{2R}{\ep}}\frac{r^3}{(r^6+1)^{\frac{2}{3}}}\mathrm{d} r\\
&=Q|B_1(0)|\ep^2\log\ep \int_{0}^{(\frac{2R}{\ep})^2}\frac{r}{(r^3+1)^{\frac{2}{3}}}\mathrm{d} r\\
&=Q|B_1(0)|\ep^2\log\ep \bigg(\frac{1}{3}\Big(-\log\big(1-\frac{r}{\sqrt[3]{r^3+1}} \big)+\sqrt[3]{-1}\log\big(\frac{\sqrt[3]{-1}r}{\sqrt[3]{r^3+1}}+1\big)\\
&\ \ \ \ -(-1)^\frac{2}{3}\log\big(1-\frac{(-1)^\frac{2}{3}r}{\sqrt[3]{r^3+1}} \big)\Big)\bigg)\Big|_0^{\frac{4R^2}{\ep^2}}.\\
\end{aligned}
\end{equation}
In the last equality above, we have used the fact that:
\begin{equation*}
\begin{aligned}
\int\frac{r}{(r^3+1)^{\frac{2}{3}}}\mathrm{d} r
&=\frac{1}{3}\Big(-\log\big(1-\frac{r}{\sqrt[3]{r^3+1}} \big)+\sqrt[3]{-1}\log\big(\frac{\sqrt[3]{-1}r}{\sqrt[3]{r^3+1}}+1\big)\\
&\ \ \ \ -(-1)^\frac{2}{3}\log\big(1-\frac{(-1)^\frac{2}{3}r}{\sqrt[3]{r^3+1}} \big)\Big)+C.
\end{aligned}
\end{equation*}
Here, for  multivalued functions, we select a particular single-valued branch. Therefore, by using Taylor expansion, we obtain
\begin{equation}\label{a63}
\begin{aligned}
E_6&\geq Q|B_1(0)|\ep^2\log\ep \Big(-\frac{1}{3}\log\big(1-\frac{r}{\sqrt[3]{r^3+1}} \big)\Big|_{r=\frac{4R^2}{\ep^2}}+O(1)\Big)\\
&=Q|B_1(0)|\ep^2\log\ep \Big(-\frac{1}{3}\log\big(1-\sqrt[3]{\frac{1}{t^3+1}} \big)\Big|_{t=\frac{\ep^2}{4R^2}}+O(1)\Big)\\
&=Q|B_1(0)|\ep^2\log\ep \Big(-\frac{1}{3}\log\big(\frac{1}{3}t^3\big)\Big|_{t=\frac{\ep^2}{4R^2}}+O(1) \Big)\\
&=-\frac{1}{3}Q|B_1(0)|\ep^2\log\ep\log\big(\frac{\ep^6}{192R^6}\big) +O(1)Q|B_1(0)|\ep^2\log\ep\\
&=-2Q|B_1(0)|\ep^2\log^2\ep+\big(\frac{1}{3}Q|B_1(0)|\log(192R^6) +O(1)Q|B_1(0)|\big)\ep^2\log\ep.
\end{aligned}
\end{equation}
In addition, by $(\ref{a53})$, we obtain
\begin{equation}\label{a64}
\begin{aligned}
E_7&\geq 2\ep^2\log\ep\int_{B_{2R}(0)}|\psi^2-1| \frac{1}{\big((|x|^3+\ep^3)^2+|y|^2\big)^{\frac{2}{3}}}\mathrm{d} z\\
&\geq C\ep^2\log\ep \int_{B_{2R}(0)}\frac{\ep}{\big((|x|^3+\ep^3)^2+|y|^2\big)^{\frac{1}{6}}}\frac{1} {\big((|x|^3+\ep^3)^2+|y|^2\big)^{\frac{2}{3}}}\mathrm{d} z\\
&= C\ep^2\log\ep \int_{B_{2R}(0)}\frac{\ep}{\big((|x|^3+\ep^3)^2+|y|^2\big)^{\frac{5}{6}}}\mathrm{d} z\geq C\ep^2\log\ep \int_{B_{2R}(0)}\frac{\ep}{((|x|^6+|y|^2)+\ep^6)^{\frac{5}{6}}}\mathrm{d} z\\
&= C\ep^2\log\ep \int_{0}^{2 R}\frac{\ep r^3}{(r^6+\ep^6)^{\frac{5}{6}}}\mathrm{d} r = C\ep^2\log\ep \int_{0}^{(\frac{2R}{\ep})^2}\frac{r}{(r^3+1)^{\frac{5}{6}}}\mathrm{d} r\geq C\ep^2\log\ep.
\end{aligned}
\end{equation}
Hence, we get from $(\ref{a63})$ and $(\ref{a64})$ that
\begin{equation}\label{a65}
  E_3 \geq -2Q|B_1(0)|\ep^2\log^2\ep+\big(\frac{1}{3}Q|B_1(0)|\log(192R^6)+C)\ep^2\log\ep.
\end{equation}
As for $E_4$, since $|\psi^2 \log \psi^2|\leq C$, we always have $|E_4 \log \ep| \leq -CE_6$. Then we obtain 
\begin{equation}\label{a66}
  |E_4| \leq C\ep^2|\log\ep|.
\end{equation}
Finally, we provide an upper bound estimate for $E_5$, where the estimation for $E_8$ is more complex and requires a separate estimation of the leading term and the residual term.
\begin{equation*}
\begin{aligned}
E_5=\int_{\Omega}\eta^{2}\ep^{2}\frac{\log\big((|x|^3+\ep^3)^2+|y|^2\big)^{\frac{2}{3}}} {\big((|x|^3+\ep^3)^2+|y|^2\big)^{\frac{2}{3}}}\mathrm{d} z+\int_{\Omega}\eta^{2} (\psi^{2}-1) \ep^{2}\frac{\log\big((|x|^3+\ep^3)^2+|y|^2\big)^{\frac{2}{3}}} {\big((|x|^3+\ep^3)^2+|y|^2\big)^{\frac{2}{3}}}\mathrm{d} z =:E_8+E_9,
\end{aligned}
\end{equation*}
where
\begin{equation*}
\begin{aligned}
E_8&=\int_{\Omega}\eta^{2}\ep^{2}\frac{\log\big((|x|^3+\ep^3)^2+|y|^2\big)^{\frac{2}{3}}} {\big((|x|^3+\ep^3)^2+|y|^2\big)^{\frac{2}{3}}}\mathrm{d} z\\
&=\int_{\Omega}\eta^{2}\ep^{2}\frac{\log(|x|^6+|y|^2+\ep^6)^ {\frac{2}{3}}}{(|x|^6+|y|^2+\ep^6)^{\frac{2}{3}}}\mathrm{d} z + \Big( \int_{\Omega} \eta^{2} \varepsilon^{2} \frac{\log \big((|x|^{3}+\varepsilon^{3})^{2}+|y|^{2}\big)^{\frac{2}{3}}} {\big((|x|^{3}+\varepsilon^{3})^{2}+|y|^{2}\big)^{\frac{2}{3}}}\mathrm{d} z\\
&\ \ \ \ -\int_{\Omega} \eta^{2} \varepsilon^{2} \frac{\log (|x|^{6}+|y|^{2}+\varepsilon^{6})^{\frac{2}{3}}} {(|x|^{6}+|y|^{2}+\varepsilon^{6})^{\frac{2}{3}}}\mathrm{d} z \Big)\\
&=:E_{10}+E_{11}.
\end{aligned}
\end{equation*}
Using $(\ref{a60})$ again, for $E_{10}$, we obtain
\begin{equation}\label{a67}
\begin{aligned}
E_{10}&=\int_{\Omega}\eta^{2}\ep^{2}\frac{\log(|x|^6+|y|^2+\ep^6)^ {\frac{2}{3}}}{(|x|^6+|y|^2+\ep^6)^{\frac{2}{3}}}\mathrm{d} z \leq\ep^2\int_{B_{R}(0)}\frac{\log((|x|^6+|y|^2)+\ep^6)^ {\frac{2}{3}}}{((|x|^6+|y|^2)+\ep^6)^{\frac{2}{3}}}\mathrm{d} z\\
&=Q|B_1(0)|\ep^2\int_{0}^{R}\frac{r^3\log(r^6+\ep^6)^{\frac{2}{3}}}{(r^6+\ep^6)^ {\frac{2}{3}}}\mathrm{d} r\leq Q|B_1(0)|\ep^2\int_{0}^{R}\frac{r^3\log(r^4+\ep^4)}{r^4+\ep^4}\mathrm{d} r\\
&=Q|B_1(0)|\ep^2\big(\frac{1}{8}\log^2(r^4+\ep^4)\big)\big|_0^R\\
&=-2Q|B_1(0)|\ep^2\log^2\ep+\big(\frac{1}{8}Q|B_1(0)|\log^2(R^4+\ep^4) \big)\ep^2,
\end{aligned}
\end{equation}
and for $E_{11}$, we get
\begin{equation*}
\begin{aligned}
E_{11}&=\int_{\Omega} \eta^{2} \varepsilon^{2} \frac{\log \big((|x|^{3}+\varepsilon^{3})^{2}+|y|^{2}\big)^{\frac{2}{3}}}{\big((|x|^{3} +\varepsilon^{3})^{2}+|y|^{2}\big)^{\frac{2}{3}}}\mathrm{d} z -\int_{\Omega} \eta^{2} \varepsilon^{2} \frac{\log (|x|^{6}+|y|^{2}+\varepsilon^{6})^{\frac{2}{3}}}{(|x|^{6}+|y|^{2}+\varepsilon^{6})^ {\frac{2}{3}}}\mathrm{d} z \\
&=\int_{\Omega} \eta^{2} \varepsilon^{2} \log (|x|^{6}+|y|^{2}+\varepsilon^{6})^{\frac{2}{3}}\Big(\frac{1}{((|x|^{3}+\varepsilon^{3}) ^{2}+|y|^{2})^{\frac{2}{3}}}-\frac{1}{(|x|^{6}+|y|^{2}+\varepsilon^{6}) ^{\frac{2}{3}}}\Big )\mathrm{d} z \\
&\ \ \ \ +\int_{\Omega} \frac{\eta^{2} \varepsilon^{2} (\log \big((|x|^{3}+\varepsilon^{3})^{2}+|y|^{2}\big)^{\frac{2}{3}}-\log (|x|^{6}+|y|^{2}+\varepsilon^{6})^{\frac{2}{3}})}{\big((|x|^{3}+\varepsilon^{3})^{2}+|y|^{2}\big) ^{\frac{2}{3}}}\mathrm{d} z\\
&=\int_{\Omega} \eta^{2} \varepsilon^{2} \log(|x|^{6}+|y|^{2}+\varepsilon^{6})^{\frac{2}{3}}\frac{(|x|^{6}+|y|^{2} +\varepsilon^{6})^{\frac{2}{3}}-\big((|x|^{3}+\varepsilon^{3})^{2}+|y|^{2}\big) ^{\frac{2}{3}}}{\big((|x|^{3}+\varepsilon^{3})^{2}+|y|^{2}\big) ^{\frac{2}{3}}(|x|^{6}+|y|^{2}+\varepsilon^{6})^{\frac{2}{3}}}\mathrm{d} z\\
&\ \ \ \ +\int_{\Omega} \frac{\eta^{2} \varepsilon^{2} \log (\frac{(|x|^{3}+\varepsilon^{3})^{2}+|y|^{2}}{|x|^{6}+|y|^{2}+\varepsilon^{6}})^{\frac{2}{3}}}{\big((|x|^{3}+\varepsilon^{3})^{2}+|y|^{2}\big)^{\frac{2}{3}}} \mathrm{d} z.
\end{aligned}
\end{equation*}
To facilitate computation, we directly estimate the magnitude of $|E_{11}|$.
\begin{equation}\label{a68}
\begin{aligned}
|E_{11}|&\leq \int_{B_{2R}(0)}\frac{\big|\varepsilon^{2} \log((|x|^{6}+|y|^{2})+\varepsilon^{6})^{\frac{2}{3}} \big|(2|x|^{3}\ep^3)^{\frac{2}{3}}}{((|x|^{6}+|y|^{2}) +\varepsilon^{6})^{\frac{4}{3}}}\mathrm{d} z+\int_{B_{2R}(0)} \frac{\ep^2\log2^{\frac{2}{3}}}{((|x|^{6}+|y|^{2})+\varepsilon^{6})^{\frac{2}{3}}}\mathrm{d} z\\
&\leq 2^{\frac{2}{3}}Q|B_1(0)|\ep^2 \int_{0}^{2R}\frac{\ep^2 r^3 r^2\big|\log(r^6+\ep^6)^{\frac{2}{3}}\big|}{(r^6+\ep^6)^{\frac{4}{3}}} \mathrm{d} r +\frac{2}{3}Q|B_1(0)|(\log 2)\ep^2\int_{0}^{2R}\frac{r^3}{(r^6+\ep^6)^{\frac{2}{3}}}\mathrm{d} r\\
&\leq 2^{\frac{2}{3}}Q|B_1(0)|\ep^2 \int_{0}^{2R}\frac{\ep^2 r^5\big|\log(C(r^4+\ep^4))\big|}{r^8+\ep^8}\mathrm{d} r+\frac{2}{3}Q|B_1(0)|(\log 2)\ep^2\int_{0}^{2R}\frac{r^3}{C(r^4+\ep^4)}\mathrm{d} r\\
&=C\ep^2\int_{0}^{2R}\frac{\ep^2 r^5\big|\log(r^4+\ep^4)\big|}{r^8+\ep^8}\mathrm{d} r +C\ep^2\int_{0}^{2R}\frac{r^3}{r^4+\ep^4}\mathrm{d} r +C\ep^2\int_{0}^{2R}\frac{\ep^2 r^5}{r^8+\ep^8}\mathrm{d} r\\
&=C\ep^2\int_{0}^{\frac{2R}{\ep}}\frac{ r^5\big|\log(\ep^4(r^4+1))\big|}{r^8+1}\mathrm{d} r +C\ep^2\int_{0}^{\frac{2R}{\ep}}\frac{r^3}{r^4+1}\mathrm{d} r +C\ep^2\int_{0}^{\frac{2R}{\ep}}\frac{ r^5}{r^8+1}\mathrm{d} r\\
&=C\ep^2|\log\ep|+O(\ep^2).
\end{aligned}
\end{equation}
The combination of equations $(\ref{a67})$ and $(\ref{a68})$ forms an estimation of $E_8$, and the lower bound estimation of $\int_{\Omega} u_{\varepsilon}^{2} \log u_{\varepsilon}^{2}\mathrm{d} z$ can be completed by solely estimating $E_9$.
\begin{equation}\label{a70}
\begin{aligned}
E_9&\leq \int_{\Omega} \eta^{2}\frac{C\ep}{\big((|x|^{3}+\ep^{3})^{2}+|y|^{2}\big) ^{\frac{1}{6}}}\frac{\ep^2\log\big((|x|^{3}+\ep^{3})^{2}+|y|^{2}\big) ^{\frac{2}{3}}}{\big((|x|^{3}+\ep^{3})^{2}+|y|^{2}\big)^{\frac{2}{3}}}\mathrm{d} z\\
&\leq C\ep^2\int_{B_R(0)}\frac{\ep\log\big((|x|^{3}+\ep^{3})^{2}+|y|^{2}\big) ^{\frac{2}{3}}}{\big((|x|^{3}+\ep^{3})^{2}+|y|^{2}\big)^{\frac{5}{6}}}\mathrm{d} z\\
&\leq C\ep^2\int_{B_R(0)}\frac{\ep\log\big(2(|x|^{6}+|y|^{2}+\ep^{6})\big) ^{\frac{2}{3}}}{\big(2(|x|^{6}+|y|^{2}+\ep^{6})\big)^{\frac{5}{6}}}\mathrm{d} z\\
&=C\ep^2\int_{0}^{R}\frac{\ep r^3\log(r^6+\ep^6)^{\frac{2}{3}}}{(r^6+\ep^6)^{\frac{5}{6}}}\mathrm{d} r +C\ep^2\int_{0}^{R}\frac{\ep r^3}{(r^6+\ep^6)^{\frac{5}{6}}}\mathrm{d} r\\
&\leq C\ep^2 \int_{0}^{R}\frac{\ep r^3\log(r^4+\ep^4)}{(r^5+\ep^5)}\mathrm{d} r +C\ep^2\int_{0}^{R}\frac{\ep r^3}{(r^6+\ep^6)^{\frac{5}{6}}}\mathrm{d} r\\
&=C\ep^2 \int_{0}^{\frac{R}{\ep}}\frac{r^3\log\big(\ep^4(1+r^4)\big)}{r^5+1}\mathrm{d} r +C\ep^2 \int_{0}^{\frac{R}{\ep}}\frac{r^3}{(r^6+1)^{\frac{5}{6}}}\mathrm{d} r\\
&=C\ep^2\log\ep\int_{0}^{\frac{R}{\ep}}\frac{r^3}{r^5+1}\mathrm{d} r +C\ep^2 \int_{0}^{\frac{R}{\ep}}\frac{r^3\log(1+r^4)}{r^5+1}\mathrm{d} r +C\ep^2 \int_{0}^{\frac{R}{\ep}}\frac{r^3}{(r^6+1)^{\frac{5}{6}}}\mathrm{d} r\\
&=O(\ep^2\log\ep)+O(\ep^2)=O(\ep^2\log\ep).
\end{aligned}
\end{equation}
Bonded with $(\ref{a62})$, $(\ref{a65})$--$(\ref{a70})$, we finally have
$$\int_{\Omega} u_{\varepsilon}^{2} \log u_{\varepsilon}^{2}\mathrm{d} z \geq \big(-\frac{1}{3}Q|B_1(0)|\log(192R^6)-C\big)\ep^2|\log\ep|.$$

\vspace{0.3cm}
{\bf{Case 3: The case for $Q=4$, $m=1$, $n=2$, $\alpha=\frac{1}{2}$.}}\vspace{0.15cm}

In this case, the homogeneous norm $d_{\alpha}(z)=(|x|^3+|y|^2)^{\frac{1}{3}}$ and
$$U_\ep(x,y)=\frac{\ep}{\big((|x|^{\frac{3}{2}}+\ep^{\frac{3}{2}})^2 +|y|^2\big)^{\frac{2}{3}}}\psi\big(|\frac{\ep^{\frac{3}{2}}(|x|^{\frac{3}{2}} +\ep^{\frac{3}{2}},y)}{(|x|^{\frac{3}{2}}+\ep^{\frac{3}{2}})^2+|y|^2} -(\frac{1}{2},0)|\big),$$
where $\psi(r)>0$, $\psi(\frac{1}{2})=1$, $\psi'(0)=0$, $\psi \in C^2[0,\frac{1}{2})\cap C^0[0,\frac{1}{2}]$ satisfies
\begin{equation}\label{a54}
  \lim_{r \to (\frac{1}{2})^-}(\frac{1}{4}-r^2)^{\frac{1}{3}}\psi'(r)=0.
\end{equation}
Lemma $\ref{lem11}$ and $(\ref{a54})$ implies that $\psi$ is $\frac{2}{3}$-H\"older continuous, thus indicating that
\begin{equation}\label{a55}
  \begin{aligned}
  \Big|\psi^2\big(|\frac{\ep^{\frac{3}{2}}(|x|^{\frac{3}{2}}+\ep^{\frac{3}{2}},y)}{(|x|^{\frac{3}{2}}+\ep^{\frac{3}{2}})^2+|y|^2}-(\frac{1}{2},0)|\big)-1\Big|&=C\Big|\psi\big(|\frac{\ep^{\frac{3}{2}}(|x|^{\frac{3}{2}}+\ep^{\frac{3}{2}},y)}{(|x|^{\frac{3}{2}}+\ep^{\frac{3}{2}})^2+|y|^2}-(\frac{1}{2},0)|\big)-\psi\big(|(\frac{1}{2},0)|\big)\Big|\\
  &\leq C\Big|\big|\frac{\ep^{\frac{3}{2}}(|x|^{\frac{3}{2}} +\ep^{\frac{3}{2}},y)}{(|x|^{\frac{3}{2}}+\ep^{\frac{3}{2}})^2+|y|^2} -(\frac{1}{2},0)\big|-\big|(\frac{1}{2},0)\big|\Big|^{\frac{2}{3}}\\
  &\leq C\big|\frac{\ep^{\frac{3}{2}}(|x|^{\frac{3}{2}}+
  \ep^{\frac{3}{2}},y)}{(|x|^{\frac{3}{2}}+\ep^{\frac{3}{2}})^2+|y|^2}\big|^{\frac{2}{3}}\\
  &=C\frac{\ep}{\big((|x|^{\frac{3}{2}}+\ep^{\frac{3}{2}})^2+|y|^2\big)^{\frac{1}{3}}}.
  \end{aligned}
\end{equation}
Referring to the computational methods presented in $(\ref{a000})$ and $(\ref{a62})$, we can readily obtain
$$\int_{\Omega} u_{\varepsilon}^{2} \log u_{\varepsilon}^{2} \mathrm{d} z =\int_{\Omega} \eta^{2} \log \eta^{2} U_{\varepsilon}^{2} \mathrm{d} z+\int_{\Omega} \eta^{2} U_{\varepsilon}^{2} \log U_{\varepsilon}^{2} \mathrm{d} z=:F_{1}+F_{2}$$
and
\begin{equation}\label{a001}
  F_{1}=O(\varepsilon^{2}).
\end{equation}
Similarly,
\begin{equation*}
  \begin{aligned}
   F_2 & =\int_{\Omega} \eta^{2} \psi^{2} \ep^{2} \frac{1}{\big((|x|^{\frac{3}{2}}+\ep^{\frac{3}{2}})^2+|y|^2\big)^{\frac{4}{3}}} \log \ep^{2}\mathrm{d} z+\int_{\Omega} \eta^{2} \psi^{2} \varepsilon^{2} \frac{1}{\big((|x|^{\frac{3}{2}}+\ep^{\frac{3}{2}})^2+|y|^2\big)^{\frac{4}{3}}} \log \psi^{2}\mathrm{d} z \\
   &\ \ \ \ -\int_{\Omega} \eta^{2} \psi^{2} \varepsilon^{2} \frac{1}{\big((|x|^{\frac{3}{2}}+\ep^{\frac{3}{2}})^2+|y|^2\big)^{\frac{4}{3}}} \log \big((|x|^{\frac{3}{2}}+\ep^{\frac{3}{2}})^2+|y|^2\big)^{\frac{4}{3}} \mathrm{d} z. \\
   &=:F_{3}+F_{4}-F_{5}.
   \end{aligned}
\end{equation*}
We still decompose $\psi^2$ into $(\psi^2-1)+1$ to leverage the $\frac{2}{3}$-H\"older continuity of $\psi$.
\begin{equation*}
  \begin{aligned}
  F_3=\int_{\Omega} \frac{\eta^{2} \ep^{2}\log \ep^{2}}{\big((|x|^{\frac{3}{2}}+\ep^{\frac{3}{2}})^2+|y|^2\big) ^{\frac{4}{3}}} \mathrm{d} z +\int_{\Omega} \frac{\eta^{2}(\psi^{2}-1)\ep^{2}\log \ep^{2}}{\big((|x|^{\frac{3}{2}}+\ep^{\frac{3}{2}})^2+|y|^2\big) ^{\frac{4}{3}}}\mathrm{d} z=:F_6+F_7,
  \end{aligned}
\end{equation*}
where
\begin{equation}\label{a002}
\begin{aligned}
F_6 &\geq \ep^2\log\ep^2 \int_{B_{2R}(0)}\frac{1}{\big((|x|^{\frac{3}{2}}+\ep^{\frac{3}{2}})^2+|y|^2\big) ^{\frac{4}{3}}} \mathrm{d} z \geq \ep^2\log\ep^2 \int_{B_{2R}(0)}\frac{1}{((|x|^3+|y|^2)+\ep^3)^{\frac{4}{3}}}\mathrm{d} z\\
&=2Q|B_1(0)|\ep^2\log\ep \int_{0}^{2 R}\frac{r^3}{(r^3+\ep^3)^{\frac{4}{3}}} \mathrm{d} r \geq 2Q|B_1(0)|\ep^2\log\ep \int_{0}^{2R}\frac{r^3}{r^4+\ep^4}\mathrm{d} r\\
&=2Q|B_1(0)|\ep^2\log\ep\big(\frac{1}{4}\log(r^4+\ep^4)\big)\big|_0^{2R}\\
&=-2Q|B_1(0)|\ep^2\log^2\ep+\big(\frac{1}{2}Q|B_1(0)| \log(16R^4+\ep^4)\big)\ep^2\log\ep,
\end{aligned}
\end{equation}
and
\begin{equation}\label{a003}
\begin{aligned}
F_7&\geq \int_{B_{2R}(0)}|\psi^2-1|\ep^2\log\ep^2 \frac{1}{\big((|x|^{\frac{3}{2}}+\ep^{\frac{3}{2}})^2+|y|^2\big) ^{\frac{4}{3}}}\mathrm{d} z\\
&\geq C\ep^2\log\ep \int_{B_{2R}(0)}\frac{\ep}{\big((|x|^{\frac{3}{2}}+\ep^{\frac{3}{2}})^2+|y|^2\big) ^{\frac{1}{3}}}\frac{1}{\big((|x|^{\frac{3}{2}}+\ep^{\frac{3}{2}})^2 +|y|^2\big)^{\frac{4}{3}}}\mathrm{d} z\\
&= C\ep^2\log\ep \int_{B_{2R}(0)}\frac{\ep}{\big((|x|^{\frac{3}{2}}+\ep^{\frac{3}{2}})^2+|y|^2\big) ^{\frac{5}{3}}}\mathrm{d} z\\
&\geq C\ep^2\log\ep \int_{B_{2R}(0)}\frac{\ep}{((|x|^3+|y|^2)+\ep^3) ^{\frac{5}{3}}}\mathrm{d} z\\
&= C\ep^2\log\ep \int_{0}^{2 R}\frac{\ep r^3}{(r^3+\ep^3)^{\frac{5}{3}}} \mathrm{d} r = C\ep^2\log\ep \int_{0}^{\frac{2R}{\ep}}\frac{r^3}{(r^3+1)^{\frac{5}{3}}}\mathrm{d} r \geq C\ep^2\log\ep.
\end{aligned}
\end{equation}
Combined with $(\ref{a002})$ and $(\ref{a003})$, we have
\begin{equation}\label{a004}
  F_3 \geq -2Q|B_1(0)|\ep^2\log^2\ep+\big(\frac{1}{2}Q|B_1(0)|\log(16R^4+\ep^4)+C\big)\ep^2\log\ep.
\end{equation}
Given $F_4$, it follows that $|F_4\log\ep|\leq C|F_6|$ for some positive constants $C$. Thus, we can derive that
\begin{equation}\label{a005}
  |F_4| \geq C\ep^2|\log\ep|.
\end{equation}
Regarding $F_5$, we still perform a decomposition showing as
\begin{equation*}
\begin{aligned}
F_5=\int_{\Omega}\eta^{2}\ep^{2}\frac{\log\big((|x|^{\frac{3}{2}} +\ep^{\frac{3}{2}})^2+|y|^2\big)^{\frac{4}{3}}}{\big((|x|^{\frac{3}{2}} +\ep^{\frac{3}{2}})^2+|y|^2\big)^{\frac{4}{3}}}\mathrm{d} z+\int_{\Omega}\eta^{2} (\psi^{2}-1) \ep^{2}\frac{\log\big((|x|^{\frac{3}{2}} +\ep^{\frac{3}{2}})^2+|y|^2\big)^{\frac{4}{3}}} {\big((|x|^{\frac{3}{2}}+\ep^{\frac{3}{2}})^2+|y|^2\big) ^{\frac{4}{3}}}\mathrm{d} z=:F_8+F_9,
\end{aligned}
\end{equation*}
and
$$F_8=:F_{10}+F_{11},$$
where \[F_{10}=\int_{\Omega}\eta^{2}\ep^{2}\frac{\log\big(|x|^3+|y|^2+\ep^3\big) ^{\frac{4}{3}}}{\big(|x|^3+|y|^2+\ep^3\big)^{\frac{4}{3}}}\mathrm{d} z ,\]
\[F_{11}=\int_{\Omega} \eta^{2} \varepsilon^{2} \frac{\log \big((|x|^{\frac{3}{2}}+\varepsilon^{\frac{3}{2}})^{2} +|y|^{2}\big)^{\frac{4}{3}}}{\big((|x|^{\frac{3}{2}} +\varepsilon^{\frac{3}{2}})^{2}+|y|^{2}\big)^{\frac{4}{3}}}\mathrm{d} z -\int_{\Omega} \eta^{2} \varepsilon^{2} \frac{\log (|x|^{3}+|y|^{2} +\varepsilon^{3})^{\frac{4}{3}}} {(|x|^{3}+|y|^{2}+\varepsilon^{3})^{\frac{4}{3}}}\mathrm{d} z.\]
In what follows, the equation $(\ref{a61})$ plays a crucial role once again, as it allows us to obtain
\begin{equation}\label{a006}
\begin{aligned}
F_{10}&=\int_{\Omega}\eta^{2}\ep^{2}\frac{\log\big(|x|^3+|y|^2+\ep^3\big) ^{\frac{4}{3}}}{\big(|x|^3+|y|^2+\ep^3\big)^{\frac{4}{3}}}\mathrm{d} z \leq \ep^2\int_{B_{R}(0)} \frac{\log((|x|^3+|y|^2)+\ep^3)^{\frac{4}{3}}}{((|x|^3+|y|^2)+\ep^3) ^{\frac{4}{3}}}\mathrm{d} z\\
&=Q|B_1(0)|\ep^2\int_{0}^{R}\frac{r^3\log(r^3+\ep^3)^{\frac{4}{3}}}{(r^3+\ep^3) ^{\frac{4}{3}}}\mathrm{d} r \leq Q|B_1(0)|\ep^2\int_{0}^{R}\frac{r^3\log(r+\ep)^4}{(r+\ep)^4}\mathrm{d} r\\
&=4Q|B_1(0)|\ep^2\int_{0}^{R}\frac{r^3\log(r+\ep)}{(r+\ep)^4}\mathrm{d} r\\
&=4Q|B_1(0)|\ep^2\Big(\frac{1}{36(r+\ep)^3}\big(85\ep^3+189\ep^2r+108\ep r^2+(66\ep^3+162\ep^2r+108\ep r^2)\log(r+\ep)\\
&\ \ \ \ +18(r+\ep)^3\log^2(r+\ep) \big)\Big)\Big|_0^R\\
&=-2Q|B_1(0)|\ep^2\log^2\ep-\frac{22}{3}Q|B_1(0)|\ep^2\log\ep+O(\ep^2),
\end{aligned}
\end{equation}
and
\vspace{-0.8cm}
\begin{equation*}
\begin{aligned}
F_{11}&=\int_{\Omega} \eta^{2} \varepsilon^{2} \frac{\log \big((|x|^{\frac{3}{2}}+\varepsilon^{\frac{3}{2}})^{2} +|y|^{2}\big)^{\frac{4}{3}}}{\big((|x|^{\frac{3}{2}} +\varepsilon^{\frac{3}{2}})^{2}+|y|^{2}\big)^{\frac{4}{3}}}\mathrm{d} z -\int_{\Omega} \eta^{2} \varepsilon^{2} \frac{\log (|x|^{3}+|y|^{2} +\varepsilon^{3})^{\frac{4}{3}}} {(|x|^{3}+|y|^{2}+\varepsilon^{3})^{\frac{4}{3}}}\mathrm{d} z \\
&\leq \int_{\Omega} \eta^{2} \varepsilon^{2} \frac{\log\big(2(|x|^{3}+|y|^{2}+\varepsilon^{3})\big) ^{\frac{4}{3}}}{\big((|x|^{\frac{3}{2}}+\varepsilon^{\frac{3}{2}})^{2} +|y|^{2}\big)^{\frac{4}{3}}}\mathrm{d} z-\int_{\Omega} \eta^{2} \varepsilon^{2} \frac{\log (|x|^{3}+|y|^{2}+\varepsilon^{3})^{\frac{4}{3}}}{(|x|^{3}+|y|^{2} +\varepsilon^{3})^{\frac{4}{3}}}\mathrm{d} z\\
& \leq \int_{\Omega} \eta^{2} \varepsilon^{2} \log(|x|^3+|y|^{2}+\varepsilon^3)^{\frac{4}{3}} \frac{(|x|^3+|y|^{2}+\varepsilon^3)^{\frac{4}{3}} -\big((|x|^{\frac{3}{2}}+\varepsilon^{\frac{3}{2}})^{2}+|y|^{2}\big) ^{\frac{4}{3}}}{\big((|x|^{\frac{3}{2}}+\ep^{\frac{3}{2}})^2+|y|^2\big) ^{\frac{4}{3}}(|x|^3+|y|^{2}+\varepsilon^3)^{\frac{4}{3}}}\mathrm{d} z\\
&\ \ \ \ +\int_{\Omega} \frac{\eta^{2} \varepsilon^{2} \log 2^{\frac{4}{3}}}{(|x|^{3}+|y|^{2}+\varepsilon^{3})^{\frac{4}{3}}}\mathrm{d} z =: F_{12}+F_{13}.
\end{aligned}
\end{equation*}
Using \eqref{a60}, we obtain
\begin{equation*}
\begin{aligned}
\big((|x|^{\frac{3}{2}}+\varepsilon^{\frac{3}{2}})^{2}+|y|^{2}\big) ^{\frac{4}{3}}
-(|x|^3+|y|^{2}+\varepsilon^3)^{\frac{4}{3}} &= \big(\big((|x|^{\frac{3}{2}}+\varepsilon^{\frac{3}{2}})^{2}+|y|^{2}\big) ^2\big) ^ {\frac{2}{3}} - \big((|x|^3+|y|^{2}+\varepsilon^3)^2\big) ^ {\frac{2}{3}}\\
&\leq \big(\big((|x|^{\frac{3}{2}}+\varepsilon^{\frac{3}{2}})^{2}+|y|^{2}\big) ^2 - (|x|^3+|y|^{2}+\varepsilon^3)^2 \big) ^ {\frac{2}{3}}\\
&= (4\ep^3|x|^3 + 4\ep^{\frac{3}{2}}|x|^{\frac{9}{2}}+4\ep^{\frac{3}{2}}|x|^{\frac{3}{2}}|y|^{2}+4\ep^{\frac{9}{2}}|x|^{\frac{3}{2}})^ {\frac{2}{3}}\\
&\leq C(\ep^2|x|^2 + \ep|x|^3 + \ep|x||y|^{\frac{4}{3}} + \ep^3|x|).
\end{aligned}
\end{equation*}
Thus, we obtain
\begin{equation}\label{a007}
\begin{aligned}
F_{12}&\leq -C \ep^{2} \int_{B_{2R}} \log(|x|^3+|y|^{2}+\varepsilon^3)^{\frac{4}{3}} \frac{\ep^2|x|^2 + \ep|x|^3 + \ep|x||y|^{\frac{4}{3}} + \ep^3|x|}{\big((|x|^{\frac{3}{2}}+\ep^{\frac{3}{2}})^2+|y|^2\big) ^{\frac{4}{3}}(|x|^3+|y|^{2}+\varepsilon^3)^{\frac{4}{3}}} \mathrm{d} z\\
&\leq -C \ep^{2} \int_{B_{2R}} \log(|x|^3+|y|^{2}+\varepsilon^3)^{\frac{4}{3}} \frac{\ep^2|x|^2 + \ep|x|^3 + \ep|x||y|^{\frac{4}{3}} + \ep^3|x|}{(|x|^3+|y|^{2}+\varepsilon^3)^{\frac{8}{3}}} \mathrm{d} z\\
&\leq -C \ep^{2}\Big( \int_{0}^{2R}\frac{\ep^2 r^5\log(r^3+\ep^3)^{\frac{4}{3}} }{(r^3+\ep^3)^{\frac{8}{3}}} \mathrm{d} r +\int_{0}^{2R} \frac{\ep r^6 \log(r^3+\ep^3)^{\frac{4}{3}}}{(r^3+\ep^3)^{\frac{8}{3}}} \mathrm{d} r + \int_{0}^{2R} \frac{\ep^3 r^4 \log(r^3+\ep^3)^{\frac{4}{3}}}{(r^3+\ep^3)^{\frac{8}{3}}} \mathrm{d} r \Big)\\
&\leq -C \ep^{2}\Big( \int_{0}^{2R}\frac{\ep^2 r^5\log(r^4+\ep^4)}{r^8+\ep^8} \mathrm{d} r +\int_{0}^{2R} \frac{\ep r^6 \log(r^4+\ep^4)}{r^8+\ep^8} \mathrm{d} r + \int_{0}^{2R} \frac{\ep^3 r^4 \log(r^4+\ep^4)}{r^8+\ep^8} \mathrm{d} r \Big)\\
&= -C \ep^{2} \Big( \int_{0}^{\frac{2R}{\ep}} \frac{r^5 \log(r^4+1)}{r^8 +1} \mathrm{d} r +\int_{0}^{\frac{2R}{\ep}} \frac{r^6 \log(r^4+1)}{r^8 +1} \mathrm{d} r +\int_{0}^{\frac{2R}{\ep}} \frac{r^4 \log(r^4+1)}{r^8 +1} \mathrm{d} r \Big)\\
&\ \ \ \ -C \ep^{2}\log\ep \Big(\int_{0}^{\frac{2R}{\ep}}\frac{r^5}{r^8 +1} \mathrm{d} r + \int_{0}^{\frac{2R}{\ep}}\frac{r^6}{r^8 +1} \mathrm{d} r +\int_{0}^{\frac{2R}{\ep}}\frac{r^4}{r^8 +1} \mathrm{d} r\Big)\\
&\leq -C\ep^2\log\ep +O(\ep^2).
\end{aligned}
\end{equation}
Similarly, we have $F_{13}=O(\ep^2)$. Finally, for the last term $F_9$, by \eqref{a55} and \eqref{a61}, we have
\begin{equation}\label{a008}
\begin{aligned}
|F_9|& \leq C \int_{\Omega}\frac{\eta^{2}\ep }{\big((|x|^{\frac{3}{2}}+\ep^{\frac{3}{2}})^2 +|y|^2\big)^{\frac{1}{3}}} \frac{\ep^2\big|\log\big((|x|^{\frac{3}{2}} +\ep^{\frac{3}{2}})^2+|y|^2\big)^{\frac{4}{3}}\big|}{\big((|x|^{\frac{3}{2}} +\ep^{\frac{3}{2}})^2+|y|^2\big)^{\frac{4}{3}}}\mathrm{d} z\\
&\leq C\ep^2 \int_{B_{2R}(0)}\frac{\ep \big| \log\big((|x|^{\frac{3}{2}} +\ep^{\frac{3}{2}})^2+|y|^2\big)^{\frac{4}{3}}\big|}{\big((|x|^{\frac{3}{2}} +\ep^{\frac{3}{2}})^2+|y|^2\big)^{\frac{5}{3}}}\mathrm{d} z \leq C\ep^2 \int_{B_{2R}(0)}\frac{\ep \big| \log\big((|x|^3+|y|^{2}) +\varepsilon^3\big)^{\frac{4}{3}}\big|}{\big((|x|^3+|y|^{2}) +\varepsilon^3\big)^{\frac{5}{3}}}\mathrm{d} z\\
&=C\ep^2\int_{0}^{2R}\frac{\ep r^3\big| \log(r^3+\ep^3)^{\frac{4}{3}}\big|}{(r^3+\ep^3)^{\frac{5}{3}}}\mathrm{d} r \leq C\ep^2\int_{0}^{2R}\frac{\ep r^3 \big| \log(r^4+\ep^4)\big|}{(r^5+\ep^5)}\mathrm{d} r =C \ep^2 |\log\ep|.
\end{aligned}
\end{equation}
Utilizing $(\ref{a001})$ and $(\ref{a004})-(\ref{a008})$, we receive the following lower bound estimates for the integral $\int_{\Omega} u_{\varepsilon}^{2} \log u_{\varepsilon}^{2} \mathrm{d} z$.
\begin{equation*}
\begin{aligned}
\int_{\Omega} u_{\varepsilon}^{2} \log u_{\varepsilon}^{2}\mathrm{d} z &\geq \big(-\frac{1}{2}Q|B_1(0)|\log(16R^4+\ep^4) -\frac{22}{3}Q|B_1(0)|+C\big)\ep^2|\log\ep|+O(\ep^2)\\
&\geq \big(-\frac{1}{2}Q|B_1(0)|\log(17R^4) -C\big)\ep^2|\log\ep|,
\end{aligned}
\end{equation*}
we complete the proof.
\end{proof}

\subsection{Proof of Theorem \ref{thm1.1}}
\begin{Lem}\label{Q>4max}
	If $Q \geq4$, $\lambda>0$, then there exists a function $v \in \mathcal{D}_0^1(\Omega)$ satisfying
	\begin{equation}\label{3.4}
		0<\max_{ t\geq 0}I_+(tv)<\frac{1}{Q}c_{m,n,\alpha}^{Q}.
	\end{equation}
\end{Lem}
\begin{proof}
From the definition of $I_+$ and $I_+ \in
\mathcal{C}^{1} \big(\mathcal{D}_0^{1}(\Omega),
\mathbb{R}\big)$, we have
\begin{equation}\label{b3.04}
I_+(tu_{\varepsilon})
=\frac{t^{2}}{2} \int_{\Omega}|\nabla_{\alpha}
u_{\varepsilon}|^{2}\mathrm{d}
z-\frac{t^{2_\alpha^*}}{2_\alpha^*}
\int_{\Omega}u_{\varepsilon}^{2_\alpha^*}\mathrm{d}
z-\frac{\lambda t^{2}}{2}\int_{\Omega} \big(
u_{\varepsilon}^{2} \log (t
u_{\varepsilon})^{2}-u_{\varepsilon}^{2}\big)\mathrm{d}
z,
\end{equation}
where $u_{\ep}$ is defined by $(\ref{a1.6})$, and
\begin{equation}\label{b3.05}
	\frac{d}{dt} I_+(t
u_{\varepsilon})
	=t \int_{\Omega}|\nabla_{\alpha}
u_{\varepsilon}|^{2}\mathrm{d} z-t^{2_\alpha^*-1}
\int_{\Omega}u_{\varepsilon}^{2_\alpha^*}\mathrm{d}
z-\lambda t\int_{\Omega} u_{\varepsilon}^{2}
\log (t^{2} u_{\varepsilon}^{2})\mathrm{d} z.
\end{equation}
We know, combining (\ref{b3.04}) and
(\ref{b3.05}), that there exists $t_{\ep}>0$ such
that $I_+(t_{\ep}u_{\ep})=\max \limits_{t \geq 0}
I_+(tu_{\ep})$.
Naturally, we have
 $\frac{d}{dt} I_+(t
 u_{\varepsilon})\Big|_{t=t_{\varepsilon}}=0$,
 which means that
\begin{equation}\label{c3.88}
	t_{\varepsilon} \int_{\Omega}|\nabla_{\alpha}
u_{\varepsilon}|^{2}\mathrm{d} z=
	t_{\varepsilon}^{2_\alpha^*-1}
\int_{\Omega}u_{\varepsilon}^{2_\alpha^*}\mathrm{d}
z+\lambda t_{\varepsilon} \int_{\Omega}
u_{\varepsilon}^{2} \log (t_{\varepsilon}^{2}
u_{\varepsilon}^{2})\mathrm{d} z.
\end{equation}
It follows from $(\ref{a1.71})$--$(\ref{a1.73})$,
$(\ref{a602})$ and $(\ref{c3.88})$ that, for $\ep
>0$ sufficiently small, we have
\begin{equation}\label{a3.600}
\begin{aligned}
	& \ \ \ \
c_{m,n,\alpha}^{Q}+O(\varepsilon^{Q-2})\\
	&=t_{\varepsilon}^{2_\alpha^*-2}
\int_{\Omega}u_{\varepsilon}^{2_\alpha^*}\mathrm{d}
z +\lambda \Big(\log t_\ep^2\int_{\Omega}
u_{\varepsilon}^{2} \mathrm{d} z
+\int_{\Omega} u_{\varepsilon}^{2}\log u_\ep^2
\mathrm{d} z \Big)\\
&=t_{\varepsilon}^{2_\alpha^*-2}\big(c_{m,n,\alpha}^{Q}+O(\varepsilon^{Q})\big)+\lambda
\big(\log t_\ep^2\big)\int_{\Omega}
u_{\varepsilon}^{2} \mathrm{d} z + \lambda
\int_{\Omega} u_{\ep}^{2}\log u_\ep^2
\mathrm{d} z.
\end{aligned}
\end{equation}
From the result of Lemma \ref{gt}, we have $\int_{\Omega}
u_{\varepsilon}^{2} \mathrm{d} z,  \int_{\Omega} u_{\ep}^{2}\log u_\ep^2
\mathrm{d} z\to 0$ as $\ep \to 0^+$. Thus,
\begin{equation}\label{a3.700}
	t_{\ep} \ra 1,\ \ \text{when} ~\ep \to 0^+.
\end{equation}

Let $t>0$ and $g_1(t)=\frac{t^{2}}{2}-\frac{t^{2_\alpha^*}}{2_\alpha^*}$. Then $g_1'(t)=0$ implies $t=1$. Since $g_1'(t)>0$ for $0<t<1$ and $g_1'(t)<0$ for $t>1$, thus $\max\limits_{t\in\mathbb{R}^+}g_1(t)=g_1(1)=\frac{1}{Q}$. From Lemma \ref{gt}, we obtain
\begin{equation}\label{3.47}
\begin{aligned}
\frac{t_{\ep}^{2}}{2}
\int_{\Omega}|\nabla_{\alpha} u_{\varepsilon}|^{2}
\mathrm{d}
z-\frac{t_{\ep}^{2_\alpha^*}}{2_\alpha^*}
\int_{\Omega}u_{\ep}^{2_\alpha^*}\mathrm{d} z
&=\Big(\frac{t_{\ep}^{2}}{2}-\frac{t_{\ep}^{2_\alpha^*}}{2_\alpha^*}\Big)c_{m,n,\alpha}^{Q}
+O(\ep^{Q-2})\\
&\leq \frac{c_{m,n,\alpha}^{Q}}{Q}+O(\ep^{Q-2}).
\end{aligned}
\end{equation}
Combining $(\ref{a3.700})$, \eqref{3.47}, Lemma \ref{gt}, Proposition \ref{lem100} and Lemma \ref{lem1}, we deduce that if we choose a smaller $R$, then for $\ep \to 0^{+}$,
\begin{equation}\label{m1}
\begin{aligned}
&\ \ \ \ I_+(t_{\varepsilon} u_{\varepsilon})\\
&=\frac{t_{\ep}^{2}}{2}
\int_{\Omega}|\nabla_{\alpha} u_{\varepsilon}|^{2}
\mathrm{d}
z-\frac{t_{\ep}^{2_\alpha^*}}{2_\alpha^*}
\int_{\Omega}u_{\ep}^{2_\alpha^*}\mathrm{d} z
-\frac{\lambda
t_{\varepsilon}^{2}}{2}\int_{\Omega}\big(
u_{\ep}^{2}\log (t_{\varepsilon}
u_{\ep})^{2}-u_{\ep}^{2}\big)\mathrm{d} z
\\
&\leq
\begin{cases}
 \frac{c_{m,n,\alpha}^{Q}}{Q}-C\lambda\ep^2|\log \ep|+O(\ep^2), & \mbox{if } Q>4 \\
  \frac{c_{m,n,\alpha}^{Q}}{Q}-(C(R)-C)\lambda\ep^2|\log \ep|+O(\ep^2), & \mbox{if } Q=4.
\end{cases}
\\&
< \frac{c_{m,n,\alpha}^{Q}}{Q}.
\end{aligned}
\end{equation}
 Therefore, the estimate
\eqref{3.4} holds for $Q \geq 4$ if we choose $v=u_{\ep}$ for some small $R$ and $\ep$.
\end{proof}
\begin{proof}[Proof of Theorem \ref{thm1.1}]
From Lemma \ref{mpgs}, the functional $I_+$ has a mountain pass structure. We set the mountain pass value
\begin{equation}\label{a3.8}
	c_{\lambda}:=\inf _{\gamma \in \Gamma} \max_{t \in [0, \bar{R}]} I_+(\gamma(t))>0,
\end{equation}
where the set of continuous paths connecting $0$ and $\bar{R}u_{\varepsilon}$ in $\mathcal{D}_0^1 (\Omega)$ is defined as
\begin{equation*}
	\Gamma=\big\{\gamma \in \mathcal{C}\big([0, \bar{R}], \mathcal{D}_0^1 (\Omega)\big): \ \gamma(0)=0, \gamma(\bar{R})=\bar{R}v \big\}
\end{equation*}
with $\bar{R}>0$ large enough such that $I_+(\bar{R}v)<0$ and $v$ comes from Lemma \ref{Q>4max}. It is clear that the path $\gamma_{\varepsilon}(t)=tv$, $t\in [0, \bar{R}]$ belongs to $\Gamma$.

Select a $(PS)_{c_{\lambda}}$ sequence $\{u_n \} \subset \mathcal{D}_0^1 (\Omega)$ such that
\begin{equation}\label{a3.900}
	I_+(u_{n}) \ \rightarrow  \ c_{\lambda}, \ I'_+(u_{n}) \ \rightarrow \ 0 \ \ \ ~\text{as}~n \ \rightarrow \ \infty.
\end{equation}
From Proposition $\ref{bounded}$, Lemma $\ref{Q>4max}$ and $(\ref{a3.900})$, $c_{\lambda}<\frac{1}{Q}c_{m,n,\alpha}^{Q}$ and there exists $u \in \mathcal{D}_0^1 (\Omega)$ such that $u_n \rightarrow u$ in $\mathcal{D}_0^1 (\Omega)$. Thus we have
$$I_+(u) = c_{\lambda} >0, \ \text{and} \ I'_+(u)=0.$$
From discussion for \eqref{2.5}, we can deduce that the critical point $u$ is a non-trivial non-negative ground state solution of the problem $(\ref{eq1.1})$. Theorem \ref{thm1.1} is proved.
\end{proof}
\vs{2mm}
\section{$\lambda<0$: Existence of solutions}\label{S4}
\vs{2mm}
 In this section, we shall first give the estimate of the logarithmic term for $2<Q<4$ and finish the proof of Theorem $\ref{thm1.2}$.
\subsection{$\lambda <0$, $2<Q<4$: Estimates of logarithmic term.}
\begin{Lem}\label{lem1}
    For $u_{\ep}$ defined by $(\ref{a1.6})$, if $\ep>0$ is sufficiently small, then for $2<Q<4$, we have
\begin{equation}\label{a601}
  \int_{\Omega}u_\ep^{2} \log u_\ep^{2}\mathrm{d} z \leq C\ep^{Q-2} \log \ep +O(\ep^{Q-2}).
\end{equation}
\end{Lem}
\begin{proof}
In contrast with the preceding cases in Section $\ref{S3}$, the initial step in estimation remains unchanged.
$$\int_{\Omega} u_{\varepsilon}^{2} \log u_{\varepsilon}^{2} \mathrm{d} z =\int_{\Omega} \eta^{2} \log \eta^{2} U_{\varepsilon}^{2} \mathrm{d} z+\int_{\Omega} \eta^{2} U_{\varepsilon}^{2} \log U_{\varepsilon}^{2} \mathrm{d} z=:G_{1}+G_{2}.$$
Similar to $(\ref{a58})$, we obtain
\begin{equation}\label{a100}
    G_1=O(\ep^{Q-2}).
\end{equation}
Next, we evaluate $G_2$, by Lemma $\ref{lem2.5}$, we have
$U\big(\frac{1}{\varepsilon} x, \frac{1}{\varepsilon^{\alpha+1}} y\big)\leq \frac{C}{(\ep^{-1}d_\alpha(z))^{Q-2}}$, combined it with Lemma \ref{gt}, we obtain
\begin{equation}\label{a97}
\begin{aligned}
G_2&=\int_{\Omega}\eta^{2}U_\ep^2\log\big(\ep^{2-Q}U^2\big(\frac{1}{\varepsilon} x, \frac{1}{\varepsilon^{\alpha+1}} y\big) \big) \mathrm{d} z \leq \int_{\Omega}\eta^{2}U_\ep^2\log
\big(C\ep^{Q-2}d_\alpha(z)^{4-2Q}\big) \mathrm{d} z\\
&
\leq O(\ep^{Q-2})+C \ep^{Q-2} \log \ep -(2Q-4)\int_{\Omega}\eta^{2}U_\ep^2\log d_\alpha(z)\mathrm{d} z\\
&
\leq
O(\ep^{Q-2})+C \ep^{Q-2} \log \ep - C  \int_{B_{2R}}
\frac{\ep^{Q-2}}{d_\alpha(z)^{2(Q-2)}}\log d_\alpha(z)\mathrm{d} z\\
&
\leq C \ep^{Q-2} \log \ep + O(\ep^{Q-2}).
\end{aligned}
\end{equation}
Combined with $(\ref{a100})$ and $(\ref{a97})$, we obtain
$$\int_{\Omega} u_{\varepsilon}^{2} \log u_{\varepsilon}^{2} \mathrm{d} z \leq C\ep^{Q-2}\log\ep+O(\ep^{Q-2}).$$
\end{proof}
\subsection{Proof of Theorem \ref{thm1.2}}
\begin{Lem}\label{2<Q<4max}
	If $2<Q<4$, $\lambda<0$, then there exists a function $v \in \mathcal{D}_0^1(\Omega)$ satisfying
	\begin{equation}
	\max_{ t\geq 0}I(tv)<\frac{1}{Q}c_{m,n,\alpha}^{Q}.
	\end{equation}
\end{Lem}
\begin{proof}
Similar to the proof of Lemma $\ref{Q>4max}$, there exists $t_{\ep}>0$ such that $I(t_{\ep}u_{\ep})=\max \limits_{t \geq 0} I(tu_{\ep})$ and $\frac{d}{dt} I(t u_{\varepsilon})\Big|_{t=t_{\varepsilon}}=0$. Then for $\ep >0$ sufficiently small, we can deduce that
\begin{equation}\label{b3.700}
	t_{\ep} \ra 1,\ \ \text{when} ~\ep \to 0.
\end{equation}
Combined with $(\ref{a1.73})$, $(\ref{m1})$ and $(\ref{b3.700})$, we obtain, for $2<Q<4$ and $\lambda<0$,
\begin{equation*}
\begin{aligned}
I(t_{\varepsilon} u_{\varepsilon})&\leq \frac{1}{Q}c_{m,n,\alpha}^{Q}+O(\ep^{Q-2})+\frac{\lambda t_{\varepsilon}^{2}}{2}\int_{\Omega}u_{\varepsilon}^{2}\mathrm{d} z-\frac{\lambda t_{\varepsilon}^{2}\log t_{\varepsilon}^{2}}{2}\int_{\Omega}u_{\varepsilon}^{2}\mathrm{d} z -\frac{\lambda t_{\varepsilon}^{2}}{2}\int_{\Omega}u_{\varepsilon}^{2}\log u_{\varepsilon}^{2} \mathrm{d} z\\
&= \frac{1}{Q}c_{m,n,\alpha}^{Q}+O(\ep^{Q-2}) + \frac{\lambda t_{\varepsilon}^{2}(1-\log t_{\ep}^{2})}{2}\int_{\Omega}u_{\varepsilon}^{2}\mathrm{d} z -\frac{\lambda t_{\varepsilon}^{2}}{2}\int_{\Omega}u_{\varepsilon}^{2}\log u_{\varepsilon}^{2} \mathrm{d} z\\
&\leq \frac{1}{Q}c_{m,n,\alpha}^{Q}+O(\ep^{Q-2}) + \frac{\lambda t_{\varepsilon}^{2}(1-\log t_{\ep}^{2})}{2}\big(C\ep^{Q-2} +O(\ep^2)\big) -\frac{\lambda t_{\varepsilon}^{2}}{2}\big(C\ep^{Q-2}\log\ep+O(\ep^{Q-2})\big)\\
&\leq \frac{1}{Q}c_{m,n,\alpha}^{Q}+O(\ep^{Q-2}) + O(\ep^2)-C\lambda \ep^{Q-2}\log\ep < \frac{1}{Q}c_{m,n,\alpha}^{Q},
\end{aligned}
\end{equation*}
where $\ep>0$ small enough. The proof is complete.
\end{proof}
\begin{proof}[Proof of Theorem \ref{thm1.2}]

According to Lemma \ref{mpgs}, in the case of $2<Q<4$ and $-\frac{2c_{m,n,\alpha}^{Q}}{Q|\Omega|}<\lambda <0$, the functional $I_+$ satisfies the mountain pass geometry structure. Additionally, by using the result of Lemma $\ref{2<Q<4max}$, the mountain pass value $c_{\lambda}$ as defined by \eqref{a3.8} satisfies the condition $0<c_{\lambda}<\frac{1}{Q}c_{m,n,\alpha}^{Q}$. Hence, by using Proposition $\ref{bounded}$ (c), and considering a $(PS){c_{\lambda}}$ sequence $\{u_n\}$, we can deduce the existence of a non-trivial critical point $u$ of $I_+$, which is non-negative and the weak limit of $\{u_n\}$. This concludes the proof of Theorem \ref{thm1.2}.
\end{proof}

\subsection{Proof of Theorem \ref{thm1.3}}
\vs{3mm}
\
\newline

Due to the compactness of the embedding from $\mathcal{D}_0^{1}(\Omega)$ to $L^2(\Omega)$ (cf. \cite{Kogoj2012}), the eigenvalue problem for $\Delta_{\alpha}$ is well-posed.
Denote $\mu_k$ be the $k$-th eigenvalue of the following Dirichlet eigenvalue equation
\begin{equation*}
	\left\{
	\begin{aligned}
		-\Delta_{\alpha} \varphi &=\mu \varphi,~~\quad  x\in \Omega,\\
		\varphi&=0, \quad\quad x\in \partial \Omega,
	\end{aligned}
	\right.
\end{equation*}
and let $\varphi_k$ be the associated eigenfunction satisfying $\|\varphi_k\|_{\mathcal{D}_0^{1}(\Omega)}=1$. It is well-known that $0<\mu_1\leq\mu_2\leq \mu_3\leq \cdots $ and $\mu_{k}\rightarrow \infty$ as $k \rightarrow \infty$. Here $\varphi_k$ constitutes a set of complete orthonormal basis for $\mathcal{D}_0^{1}(\Omega)$.

Next, the proof of Theorem \ref{thm1.3} will depend the following Dual Fountain Theorem, which will help us to find infinitely many weak solutions with negative energy for $-\frac{2^{\frac{Q-2}{Q}} c_{m,n,\alpha}^2}{|\Omega|^{\frac{2}{Q}} (Q-2)^{\frac{Q-2}{Q}}}<\lambda<0$ and $Q>2$.
\begin{Prop}\label{thm4.4}
(Dual Fountain Theorem)
Suppose\\
	\indent $(A)$ The compact group $G$ acts isometrically on the Banach space $X=\overline{\mathop{\bigoplus}\limits_{j \in \mathbb{N}}X_{j}}$, the spaces $X_j$ are invariant and there exists a finite-dimensional space $V$ such that, for every $j \in \mathbb{N}$, $X_j \simeq V$ and the action of $G$ on $V$ is admissible.\\
    \indent Then, under the assumption $(A)$, let $\varphi \in \mathcal{C}^1(X,\mathbb{R})$ be an invariant functional. If for every $k \geq k_0$, there exists $\rho_{k}>r_{k}>0$ such that\\
    \indent $(B_{1}).~a_{k}:=\mathop{\inf}\limits_{\substack{u \in Z_{k} \\\|u\|=\rho_{k}}} \varphi(u) \geq 0$, where $Z_{k}:=\overline{\mathop{\bigoplus}\limits_{j=k}^{\infty}X_{j}}$, \\
    \indent $(B_{2}).~b_{k}:=\mathop{\max}\limits_{\substack{u \in Y_{k} \\\|u\|=r_{k}}} \varphi(u)<0$, where $Y_{k}:=\mathop{\bigoplus}\limits_{j=0}^{k}X_{j}$,\\
    \indent $(B_{3}).~d_{k}:=\mathop{\inf}\limits_{\substack{u \in Z_{k} \\\|u\| \leq \rho_{k}}} \varphi(u) \rightarrow 0, k \rightarrow \infty$.\\
    Thus there exists a sequence of negative numbers $\{c_k\}$ such that
     $c_k \to 0$ as $k\to\infty$, and
    $\varphi$ has a $(PS)_{c_k}^{*}$ sequence.
    Furthermore, if \\
    \indent $(B_{4})$.~$\varphi$ satisfies the $(PS)_{\bar{c}}^{*}$ condition for every $\bar{c} \in [d_{k_0},0)$,\\
then $\varphi$ has a sequence of negative critical values converging to 0.

\end{Prop}
\begin{proof}
See Theorem 2 in \cite{BW1995}.
\end{proof}
Thus, we have
\begin{Lem}\label{Idfgs}
When $ \lambda<0,$ the functional $I$  satisfies the Dual Fountain geometry structure $(B_1)$--$(B_3)$ of Proposition \ref{thm4.4}. If in addition, we have $0>\lambda\geq -\frac{2^{\frac{Q-2}{Q}} c_{m,n,\alpha}^2}{|\Omega|^{\frac{2}{Q}} (Q-2)^{\frac{Q-2}{Q}}}$, $I$  also satisfies the condition $(B_4)$ in Proposition \ref{thm4.4}.
\end{Lem}
\begin{proof}
To prove Lemma $\ref{Idfgs}$, we need to verify that conditions $(B_1)$--$(B_4)$ in Proposition $\ref{thm4.4}$ hold. Denote $Y_k=\operatorname{span}\left\{\varphi_{1}, \varphi_{2}, \cdots,\varphi_{k}\right\}$ and $Z_k=\overline{\operatorname{span}\left\{\varphi_{k}, \varphi_{k+1}, \cdots\right\}}$. Then for $u\in Y_{k}$, we get that
\begin{equation}\label{t1}
\mu_k\|u\|_2^2 \geq \|u\|_{\mathcal{D}_0^{1}(\Omega)}^2
\end{equation}
and for $u\in Z_{k}$,
\begin{equation}\label{t2}
\mu_k\|u\|_2^2 \leq \|u\|_{\mathcal{D}_0^{1}(\Omega)}^2.
\end{equation}
Firstly, we verify $(B_1)$. We deduce that, for $u \in Z_{k}$, and a fixed small number $\delta>0$,
\begin{equation*}
\begin{aligned}
I(u) &\geq  \frac{1}{2}\|u\|_{\mathcal{D}_0^{1}(\Omega)}^{2}-\frac{1}{2_\alpha^*}\|u\|_{2_\alpha^*}^{2_\alpha^*}- \frac{|\lambda|}{2}\|u\|_{2}^{2}+\frac{|\lambda|}{2} \int_{\{|u|<1\}} u^{2} \log u^{2} \mathrm{d} z \\
&\geq  \frac{1}{2}\|u\|_{\mathcal{D}_0^{1}(\Omega)}^{2}-\frac{C}{2_\alpha^*}\|u\|_{\mathcal{D}_0^{1}(\Omega)}^{2_\alpha^*} -\frac{|\lambda|}{2} \mu_{k}^{-1}\|u\|_{\mathcal{D}_0^{1}(\Omega)}^{2} - \frac{|\lambda|}{2}\|u\|_{2-\delta}^{2-\delta} \\
&\geq \frac{1}{2}\|u\|_{\mathcal{D}_0^{1}(\Omega)}^{2}-\frac{C}{2_\alpha^*}\|u\|_{\mathcal{D}_0^{1}(\Omega)}^{2_\alpha^*} -\frac{|\lambda|}{2} \mu_{k}^{-1}\|u\|_{\mathcal{D}_0^{1}(\Omega)}^{2} - \frac{C|\lambda|}{2}\|u\|_2^{2-\delta} \\
&\geq \frac{1}{2}\|u\|_{\mathcal{D}_0^{1}(\Omega)}^{2}-\frac{C}{2_\alpha^*}\|u\|_{\mathcal{D}_0^{1}(\Omega)}^{2_\alpha^*} -\frac{|\lambda|}{2} \mu_{k}^{-1}\|u\|_{\mathcal{D}_0^{1}(\Omega)}^{2} - \frac{C|\lambda|}{2}\mu_k^{-\frac{2-\delta}{2}}\|u\|_{\mathcal{D}_0^{1}(\Omega)}^{2-\delta}.
\end{aligned}
\end{equation*}
Since $\mu_{k}\rightarrow \infty$ as $k \rightarrow \infty$, we may take $k>0$ large enough satisfies $\frac{|\lambda|}{2} \mu_{k}^{-1} < \frac{1}{4}$, this allows us to get
\begin{equation}\label{c1}
  I(u) \geq \frac{1}{4}\|u\|_{\mathcal{D}_0^{1}(\Omega)}^{2} - \frac{C}{2_\alpha^*}\|u\|_{\mathcal{D}_0^{1}(\Omega)}^{2_\alpha^*} - \frac{C|\lambda|}{2}\mu_k^{-\frac{2-\delta}{2}}\|u\|_{\mathcal{D}_0^{1}(\Omega)}^{2-\delta}.
\end{equation}
Consequently, there exists $\rho_k>0$ such that
for $u \in Z_{k}$, $\|u\|_{\mathcal{D}_0^{1}(\Omega)}=\rho_k$,
\begin{equation}\label{t3}
\frac{1}{4}\|u\|_{\mathcal{D}_0^{1}(\Omega)}^{2}-\frac{C}{2_\alpha^*}
\|u\|_{\mathcal{D}_0^{1}(\Omega)}^{2_\alpha^*}
=\frac{1}{4} \rho^2_k - \frac{C}{2_\alpha^*}\rho^{2_\alpha^*}_k>0,
\end{equation}
and for $k_0$ large enough and $k\geq k_0$, we have
\begin{equation}\label{t4}
I(u) \geq \frac{1}{4}\rho^2_k - \frac{C}{2_\alpha^*} \rho^{2_\alpha^*}_k - \frac{C|\lambda|}{2}\mu_k^{-\frac{2-\delta}{2}} \rho^{2-\delta}_k \geq 0,
\end{equation}
which means that $(B_{1})$ is satisfied.

\smallskip

For  any $v \in Y_{k} \backslash \{0\}$, set $w:=\frac{v}{\|v\|_{\mathcal{D}_0^{1}(\Omega)}}$. There holds
\begin{equation*}
\begin{aligned}
I(v) \leq
&~\|v\|_{\mathcal{D}_0^{1}(\Omega)}^{2}\Big(\frac{1}{2}-\frac{1}{2_\alpha^*}\|v\|_{\mathcal{D}_0^{1}(\Omega)}^{2_\alpha^*-2} \|w\|_{2_\alpha^*}^{2_\alpha^*} + \frac{|\lambda |}{2}\|w\|_{2}^{2}+\frac{| \lambda|}{2}\left(\log \|v\|_{\mathcal{D}_0^{1}(\Omega)}^{2}\right)\|w\|_{2}^2\Big)\\
&+\|v\|_{\mathcal{D}_0^{1}(\Omega)}^{2} \frac{|\lambda|}{2}\big(C_{1}\|w\|_{2-\delta}^{2-\delta} + C_{2}\|w\|_{2+\delta}^{2+\delta}\big),
\end{aligned}
\end{equation*}
for $\delta>0$ small enough.
Note that $\log \|v\|_{\mathcal{D}_0^{1}(\Omega)}\to -\infty$ as $\|v\|_{\mathcal{D}_0^{1}(\Omega)}\to 0$.
Hence, there exists a constant $r_{k} \in(0, \rho_k)$ sufficiently small, such that
$$\max _{\substack{v \in Y_{k} \\ \|v\|_{\mathcal{D}_0^{1}(\Omega)}=r_k}} I(v)<0,$$
which implies that $(B_2)$ holds.

It follows from $(\ref{c1})$ and $(\ref{t3})$ that, for $u \in Z_{k}$, $\|u\|\leq \rho_k$,
$$I(u) \geq - \frac{C|\lambda|}{2}\mu_k^{-\frac{2-\delta}{2}}\|u\|_{\mathcal{D}_0^{1}(\Omega)}^{2-\delta} \geq - \frac{C|\lambda|}{2}\mu_k^{-\frac{2-\delta}{2}}\rho_k^{2-\delta}.$$
Since $\mu_{k} \rightarrow \infty$ as $k \rightarrow \infty$, $(B_3)$ is also satisfied.

Finally, for $Q>2$ and $0>\lambda \geq -\frac{2^{\frac{Q-2}{Q}} c_{m,n,\alpha}^2}{|\Omega|^{\frac{2}{Q}} (Q-2)^{\frac{Q-2}{Q}}}$, then $\frac{1}{Q}c_{m,n,\alpha}^Q -\frac{|\Omega|}{Q}\big(\frac{Q-2}{2}\big)^{\frac{Q-2}{2}}|\lambda|^{\frac{Q}{2}}\geq 0$. Thus, for every $\bar{c} \in [d_{k_0},0)$, the condition in Proposition $\ref{bounded}$ (b) will be satisfied. From the result of Proposition $\ref{bounded}$ (b), we know that $I$ satisfies the $(PS)_{\bar{c}}^{*}$ condition, which imples that  the condition  $(B_4)$ is satisfied.  The proof of Lemma \ref{Idfgs} is completed.
\end{proof}

\begin{proof}[Proof of Theorem $\ref{thm1.3}$]
From Lemma $\ref{Idfgs}$, for $Q>2$ and $\lambda <0$, $I$ satisfies conditions $(B_1)$--$(B_3)$ in Proposition $\ref{thm4.4}$.
Therefore, there exists a sequence of negative numbers $\{c_k\}$ such that $c_k \to 0$ as $k\to\infty$, and $I$ possesses a $(PS)_{c_k}^{*}$ sequence denoted as $\{u_{l}^k\}_{l=1}^{\infty}$.

Hence, according to Proposition $\ref{bounded}$ (c), there exists $u^k \in \mathcal{D}_0^1(\Omega)$ such that $u_{l}^k \rightharpoonup u^k$ as $l\to \infty$, and $u^k\neq 0$ is a critical point of $I$.
Although we may not distinguish $u^k$ for various values of $k$, we are able to establish the existence of at least one non-trivial solution to the problem $(\ref{eq1.1})$.

Furthermore, if $\lambda \geq -\frac{2^{\frac{Q-2}{Q}} c_{m,n,\alpha}^2}{|\Omega|^{\frac{2}{Q}} (Q-2)^{\frac{Q-2}{Q}}}$, Lemma $\ref{Idfgs}$ also tells us that $I$ satisfies the condition $(B_4)$ of Proposition \ref{thm4.4}.
Then we can choose the sequence of negative numbers $\{c_k\}\subset [d_{k_0},0)$, such that $c_i\not=c_j$ if $i\not=j$, and $u_{l}^k \to u^k$ in $\mathcal{D}_0^1(\Omega)$ as $l\to \infty$, and $u^k\neq 0$ is a critical point of $I$ such that $I(u^k)=c_k$.
Thus we complete the proof of Theorem $\ref{thm1.3}$.
\end{proof}

\section{Appendix}
In this section, we prove a H\"{o}lder continuous (with respect to the homogeneous metric $d_\alpha$) result for the non-negative weak solutions and give the log-Sobolev inequality for the Baouendi-Grushin operator.

\subsection{Regularity}
\
\newline

Firstly, by the Moser iteration method, we have
\begin{Lem}\label{lem5.1}
  For any open bounded domain $\Omega,$ suppose $\lambda\in \mathbb{R}$, $u\in \mathcal{D}_0^1(\Omega)$ is a weak solution of
  \[-\Delta_{\alpha} u=|u|^{\frac{4}{Q-2}}u + \lambda u\log u^2. \] Then $u\in L^{q}(\Omega) \cap \mathcal{D}_0^1(\Omega)$ for any $q\geq 1$.
\end{Lem}
\begin{proof}
  For $M>1$ and some $s\geq 0,$ let $v=\min\{|u|, M\},$ then $G(u):=uv^{2s}\in \mathcal{D}_0^1(\Omega).$
   This is because $G(0)=0$ and $G$ is a Lipschitz piecewise smooth function (cf. Theorem 7.8 in Gilbarg and Trudinger \cite{GT1977}). Choosing $G(u)$
   as a test function, we obtain
  \begin{equation}
 \int_{\Omega}|\nabla_\alpha u|^{2}v^{2s} \mathrm{d} z +2s\int_{\{x\in \Omega: |u(x)|\leq M \}} |\nabla_{\alpha} (|u|)|^2|u|^{2s}  =\int_{\Omega}|u|^{\frac{4}{Q-2}} u^2 v^{2s} \mathrm{d} z+\lambda \int_{\Omega} u^2 v^{2s}\log u^2 \mathrm{d} z.
  \end{equation}
  By using Sobolev inequality, H\"{o}lder inequality and the fact $|u\log |u||\leq C(1+|u|^{\frac{Q+2}{Q-2}})$, 
for some $K>0$, we have
  \begin{equation}
  \begin{aligned}
  \Big(\int_{\Omega} |uv^s|^{\frac{2Q}{Q-2}} \mathrm{d} z\Big)^{\frac{Q-2}{Q}}
 & \leq 
C \int_{\Omega}|\nabla_\alpha (uv^{s})|^{2}\mathrm{d} z \\&
\leq C\int_{\Omega}|u|^{\frac{4}{Q-2}}u^2 v^{2s} \mathrm{d} z+ C \int_{\Omega}|u|^{2s+1} \mathrm{d} z\\&
\leq C   K^{\frac{4}{Q-2}}\int_{\Omega}u^2 v^{2s}\mathrm{d} z+C \int_{\Omega}|u|^{2s+1} \mathrm{d} z + C\Big(\int_{|u|>K} |u|^{\frac{2Q}{Q-2}}\mathrm{d} z \Big)^{\frac{2}{Q}} \Big(\int_{\Omega}|uv^s|^{\frac{2Q}{Q-2}}\mathrm{d} z\Big)^{\frac{Q-2}{Q}}.\\
  \end{aligned}
  \end{equation}
  Choose $K$ large enough such that $C\Big(\int_{u>K} |u|^{\frac{2Q}{Q-2}} \mathrm{d} z\Big)^{\frac{2}{Q}}\leq \frac{1}{2},$ we derive that:
  \[ \Big(\int_{\Omega} |uv^s|^{\frac{2Q}{Q-2}} \mathrm{d} z\Big)^{\frac{Q-2}{Q}}
  \leq C\int_{\Omega}u^2 v^{2s}\mathrm{d} z +C \int_{\Omega}|u|^{2s+1} \mathrm{d} z.\]
  Then by letting $M\to \infty$, we must have:
 \begin{equation}\label{Moser}
   \Big(\int_{\Omega}|u|^{\frac{2(s+1)Q}{Q-2}}\mathrm{d} z\Big)^{\frac{Q-2}{Q}}
  \leq C\int_{\Omega}|u|^{2(s+1)}\mathrm{d} z+C \int_{\Omega}|u|^{2s+1} \mathrm{d} z.
 \end{equation}
 Since $\Omega$ is bounded, if $u\in L^{2(s+1)}(\Omega),$ we can deduce from \eqref{Moser} that $u\in L^{\frac{2(s+1)Q}{Q-2}}(\Omega).$ By iteration, we conclude that  $u\in L^{q}(\Omega)$ for any $q\geq 1$.
\end{proof}
By Lemma $\ref{lem5.1}$ above, we obtain that $f(u)=u^{\frac{Q+2}{Q-2}} + \lambda u\log u^2\in L^{q}(\Omega)$ for some $q>\frac{Q}{2},$ so a standard De Giorgi iteration argument yields that:
\begin{Lem}
  For $u$ given in Lemma $\ref{lem5.1}$, we have $u\in L^{\infty}(\Omega).$
\end{Lem}
\begin{proof}
  See the proof of Theorem 3.1 in Guti\'{e}rrez and Lanconelli \cite{Lanconelli2003}.
\end{proof}
Furthermore, Theorem 5.5 in \cite{Lanconelli2003} provided a nonhomogeneous Harnack inequality. Thus, similar to the elliptic case (cf. Chapter 8 in \cite{GT1977}), we obtain the following H\"{o}lder estimate immediately.
\begin{Prop}
  The non-negative weak solution $u$ in Lemma $\ref{lem5.1}$ is locally H\"{o}lder continuous with respect to the homogeneous metric $d_\alpha$, that is, there exists some $\vartheta \in (0,1)$, $z_1, z_2 \in B_r(z_0)$ such that
\begin{equation*}
  |u(z_1)-u(z_2)|\leq C \sup_{B_{2r}(z_0)}|u|d_\alpha(z_1-z_2)^{\vartheta}\ \ \text{for}\ \ B_{2r}(z_0) \subset \Omega.
\end{equation*}
\end{Prop}

\subsection{The log-Sobolev inequality}\label{subsection5.2}
\
\newline\vspace{-0.3cm}
\begin{Lem}\label{log-sob}
 If $0\leq u \in \mathcal{D}_0^{1}(\Omega),$  then for every $\varepsilon>0,$ we have
 \begin{equation}\label{logsob1}
   \int_{\Omega}u^2(\log u)_+\mathrm{d} z \leq \varepsilon \int_{\Omega}|\nabla_{\alpha} u|^2 \mathrm{d} z+ (C-\frac{Q}{4}\log \varepsilon)\|u\|_{2}^2+
   \|u\|_{2}^2\log \|u\|_{2}^2.
 \end{equation}
\end{Lem}
\begin{proof}
  By Lemma 3.5 in \cite{GN1996}, it can be seen that if $u\in \mathcal{D}_0^{1}(\Omega),$ then $u_+, |u|, \min (u,1)\in \mathcal{D}_0^{1}(\Omega).$
  Thus the quadratic form $D[u]=\int_{\Omega}|\nabla_{\alpha} u|^2 \mathrm{d} z$ defined on the domain $\mathcal{D}_0^{1}(\Omega)$ is a Dirichlet form on $L^2(\Omega).$ Therefore, the Sobolev embedding from $\mathcal{D}_0^{1}(\Omega)$ to $L^{\frac{2Q}{Q-2}}(\Omega)$ is equivalent to the log-Sobolev inequality \eqref{logsob1} (cf. Chapter 1 and Chapter 2 in \cite{D1989}).
\end{proof}

\section*{Acknowledgements}
Hua Chen is supported by National Natural Science Foundation of China (Grant No. 12131017) and National Key R$\&$D Program of China (no. 2022YFA1005602).

\section*{Declarations}
 On behalf of all authors, the corresponding author states that there is no conflict of interest.  Our manuscript has no associated data.

\end{document}